\documentclass[a4paper,UKenglish,cleveref,thm-restate]{lipics-v2021}

\usepackage{mathtools}
\usepackage{physics}
\usepackage{graphicx}
\usepackage[dvipsnames, table]{xcolor}
\usepackage{cite}
\usepackage{bm}
\usepackage{booktabs}
\usepackage{diagbox}
\usepackage{MnSymbol}

\crefformat{footnote}{#2\footnotemark[#1]#3}

\newtheorem{question}[theorem]{Question}

\nolinenumbers

\hideLIPIcs

\DeclareGraphicsExtensions{.pdf,.png,.eps}

\definecolor{red30}{RGB}{229,242,251}
\definecolor{green30}{RGB}{132,195,237}
\definecolor{openColor}{RGB}{234,234,234}

\newcommand{\N}{\mathbb{N}}
\newcommand{\R}{\mathbb{R}}
\newcommand{\calG}{\mathcal{G}}
\newcommand{\calH}{\mathcal{H}}
\newcommand{\calF}{\mathcal{F}}
\newcommand{\calS}{\mathcal{S}}
\newcommand{\calB}{\mathcal{B}}
\newcommand{\calBc}{\mathcal{B}_{\mathrm{c}}}
\newcommand{\calP}{\mathcal{P}}
\newcommand{\calC}{\mathcal{C}}
\newcommand{\calI}{\mathcal{I}}
\newcommand{\calComp}{\mathcal{K}}
\newcommand{\calQ}{\mathcal{Q}}

\newcommand{\union}[1]{\ensuremath{\overline{#1}}}

\usepackage{xargs}
\usepackage{ifthen}

\newcommandx{\cn}[3]{%
   \ifthenelse{ \equal{#3}{} }
      {\ensuremath{\operatorname{c}_{\text{\fontfamily{iwona}\selectfont #1}}^{#2}}}
      {\ensuremath{\operatorname{c}_{\text{\fontfamily{iwona}\selectfont #1}}^{#2}(#3)}}
}

\newcommand{\setmid}{\ensuremath{\; \colon \;}}
\newcommand{\shortref}[1]{{\color{black}$\langle$\ref{#1}$\rangle$}}
\newcommand{\smallshortref}[1]{{\scriptsize{\color{black}$\langle$\ref{#1}$\rangle$}}}
\newcommand{\ra}[1]{\renewcommand{\arraystretch}{#1}}
\newcolumntype{Y}{@{\extracolsep{3pt}}>{\centering\arraybackslash}X@{\extracolsep{0pt}}}

\DeclarePairedDelimiterX\set[1]\lbrace\rbrace{\def\given{\setmid}#1}
\DeclarePairedDelimiter\ceil\lceil\rceil

\DeclareMathOperator{\tw}{tw}
\DeclareMathOperator{\mad}{mad}
\DeclareMathOperator{\shift}{Sh}

\title{Boundedness and Separation in the Graph Covering Number Framework}

\author{Miriam Goetze}{Karlsruhe Institute of Technology, Germany}{miriam.goetze@kit.edu}{https://orcid.org/0000-0001-8746-522X}{funded by the Deutsche Forschungsgemeinschaft (DFG, German Research Foundation) -- 520723789}

\author{Peter Stumpf}{Charles University, Prague, Czech Republic\\ Czech Technical University in Prague, Czech Republic}{stumpf@kam.mff.cuni.cz}{https://orcid.org/0000-0003-0531-9769}{funded by Czech Science Foundation grant no. 23-04949X}

\author{Torsten Ueckerdt}{Karlsruhe Institute of Technology, Germany}{torsten.ueckerdt@kit.edu}{https://orcid.org/0000-0002-0645-9715}{supported by the Deutsche Forschungsgemeinschaft (DFG, German Research Foundation) -- 520723789}

\authorrunning{M. Goetze, P. Stumpf, T. Ueckerdt}

\Copyright{Miriam Goetze, Peter Stumpf, Torsten Ueckerdt}



\begin{document}

\maketitle

\begin{abstract}
    For a graph class $\calG$ and a graph $H$, the four $\calG$-covering numbers of $H$, namely global $\cn{g}{\calG}{H}$, union $\cn{u}{\calG}{H}$, local $\cn{l}{\calG}{H}$, and folded $\cn{f}{\calG}{H}$, each measure in a slightly different way how well $H$ can be covered with graphs from $\calG$.
    For every $\calG$ and $H$ it holds
    \[
        \cn{g}{\calG}{H} \geq \cn{u}{\calG}{H} \geq \cn{l}{\calG}{H} \geq \cn{f}{\calG}{H}
    \]
    and in general each inequality can be arbitrarily far apart.
    We investigate structural properties of graph classes~$\calG$ and $\calH$ such that for all graphs $H \in \calH$, a larger $\calG$-covering number of $H$ can be bounded in terms of a smaller $\calG$-covering number of $H$.
    For example, we prove that if $\calG$ is hereditary and the chromatic number of graphs in $\calH$ is bounded, then there exists a function $f$ (called a binding function) such that for all $H \in \calH$ it holds $\cn{u}{\calG}{H} \leq f(\cn{l}{\calG}{H})$.

    For $\calG$ we consider graph classes that are component-closed, hereditary, monotone, sparse, or of bounded chromatic number.
    For $\calH$ we consider graph classes that are sparse, $M$-minor-free, of bounded chromatic number, or of bounded treewidth.
    For each combination and every pair of $\calG$-covering numbers, we either give a binding function $f$ or provide an example of such $\calG,\calH$ for which no binding function exists.
\end{abstract}

\section{Introduction}
\label{sec:introduction}

One of the most fundamental concepts in graph theory is the decomposition of a (large) graph~$H$ into smaller pieces.
What ``decomposing'' means exactly and what ``small pieces'' are precisely depends on the concrete situation and gives room for a variety of problems.
For example, there are numerous variants of graph colorings, where one colors the vertices or edges of~$H$ subject to each color class forming restricted subgraphs such as independent sets, matchings, forests, planar graphs, interval graphs, $K_n$-free graphs, just to name a few.

In 2016, Knauer and the last author~\cite{Knauer2016_3w3c1g} introduced \emph{covering numbers} as a unified framework to investigate decomposition and coloring problems and their interrelations.
Specifically, they introduced for any graph~$H$, called a \emph{host graph}, and any graph class~$\calG$, called a \emph{guest class}, the global, local, and folded $\calG$-covering number of~$H$~\cite{Knauer2016_3w3c1g}, which were later complemented by the union $\calG$-covering number of~$H$~\cite{blasius2018local_boxicity,merker2019local_page_numbers}.

Each of the four $\calG$-covering numbers is an integer-valued parameter of the host graph~$H$ and each measures in a different way how well~$H$ can be decomposed into (in general: covered by) graphs from the guest class~$\calG$.
For example\footnote{
    We give the formal definitions of all four $\calG$-covering numbers in \cref{subsec:covering-numbers} below.
    \label{note:formal-definitions}
}, the global $\calG$-covering number of~$H$, denoted by~$\cn{g}{\calG}{H}$, is the smallest~$k$ for which there are $k$ subgraphs of~$H$ that are in graph class~$\calG$ and together cover all edges of~$H$.
Many important graph parameters can be directly interpreted as a global $\calG$-covering number for an appropriate guest class~$\calG$.
One of the earliest instances is the \emph{arboricity} of~$H$, i.e., the minimum number of forests that decompose~$H$, or in other words, the global $\calF$-covering number~$\cn{g}{\calF}{H}$ of~$H$ where $\calF$ denotes the class of all forests.
Famously, Nash-Williams proved~\cite{Nash-William1964_decomposition_into_forests} that for every graph~$H$ we have
\[
    \cn{g}{\calF}{H} = \max\left\{ \ceil*{\frac{\abs{E(H')}}{\abs{V(H')}-1}} \setmid H' \subseteq H, \abs{V(H')} > 1 \right\}.
\]
So for the global $\calG$-covering number we cover the host graph~$H$ with subgraphs from $\calG$.
Relaxing step-by-step the concept of a $\calG$-cover results\cref{note:formal-definitions} in the union, local, and folded $\calG$-covering numbers of~$H$, denoted by $\cn{u}{\calG}{H}$, $\cn{l}{\calG}{H}$, $\cn{f}{\calG}{H}$, respectively.
With each $\calG$-cover being less restrictive than the previous, we obtain a chain of inequalities
\begin{equation}
    \cn{g}{\calG}{H} \geq \cn{u}{\calG}{H} \geq \cn{l}{\calG}{H} \geq \cn{f}{\calG}{H}
    \label{eq:gulf-chain}
\end{equation}
that holds for every host graph $H$ and every guest class $\calG$.

Depending on the host graph $H$ and the guest class $\calG$, each of the three inequalities in~\eqref{eq:gulf-chain} may hold with equality, but may also be arbitrarily far apart.
(We shall see several such ``separating examples'' in the course of this paper.)
That is, in general, there is no functional relation between the four $\calG$-covering numbers, unless we have additional information about $H$ and/or $\calG$.
In this paper, we investigate under which conditions on $H$ and $\calG$ two different $\calG$-covering numbers are functionally related.
Specifically, we ask (for example) for a guest class $\calG$ and a class $\calH$ of host graphs whether there exists a function $f$ such that for all graphs $H \in \calH$ we have $\cn{u}{\calG}{H} \leq f(\cn{l}{\calG}{H})$, i.e., whether within graphs in $\calH$ the union $\calG$-covering number can be bounded in terms of the local $\calG$-covering number.

\subsection{Covering Numbers}
\label{subsec:covering-numbers}

It is convenient to formalize a covering in terms of a graph homomorphism.
For graphs $G$ and $H$, a \emph{graph homomorphism} is a map $\varphi \colon V(G) \to V(H)$ with the property that if $uv \in E(G)$, then $\varphi(u)\varphi(v) \in E(H)$, i.e., $\varphi$ maps vertices of $G$ (not necessarily injectively) to vertices of $H$ such that edges are mapped to edges.
For abbreviation we shall simply write $\varphi \colon G \to H$ instead of $\varphi \colon V(G) \to V(H)$.
For a host graph $H$, a guest class $\calG$ and a positive integer $t$, a \emph{$t$-global $\calG$-cover} of $H$ is an edge-surjective homomorphism $\varphi \colon G_1 \cupdot \cdots \cupdot G_t \to H$ such that $G_i \in \calG$ for each $i \in [t]$.\footnote{For $t \in \N$, we denote $[t] = \{1,\ldots,t\}$ to be the set of the first $t$ counting numbers.}
Here $\cupdot$ denotes the vertex-disjoint union of graphs.
We say that $\varphi$ is \emph{injective} if its restriction $\varphi|_{G_i}$ to $G_i$ is injective for each $i \in [t]$.
A $\calG$-cover $\varphi$ is called \emph{$s$-local} if $|\varphi^{-1}(v)| \leq s$ for every $v \in V(H)$.

\medskip

In other words, if $\varphi$ is a $\calG$-cover of $H$, then
\begin{itemize}
 \item $\varphi$ is $t$-global if it uses at most $t$ graphs from the guest class $\calG$,
 \item $\varphi$ is injective if $\varphi(G_i)$ is a copy\footnote{A \emph{copy of $G$ in $H$} is a subgraph $H' \subseteq H$ such that $H'$ is isomorphic to $G$.} of $G_i$ in $H$ for each $i \in [t]$,
 \item $\varphi$ is $s$-local if for each $v \in V(H)$ at most $s$ vertices are mapped onto $v$.
\end{itemize}

\subparagraph*{The global \texorpdfstring{$\bm{\calG}$}{G}-covering number of $\bm{H}$}
is
\[
 \cn{g}{\calG}{H} = \min\{t \in \N \setmid \text{there exists a $t$-global injective $\calG$-cover of $H$}\}.  
\]
Global $\calG$-covering numbers for different guest classes $\calG$ appear frequently in the literature.
This includes the above-mentioned arboricity~\cite{Nash-William1964_decomposition_into_forests} and variants thereof~\cite{AK84,FG78,AEH81,AA89}, edge-clique cover number~\cite{EGP66}, bipartite dimension~\cite{FH96}, thickness~\cite{Tut63}, boxicity~\cite{Rob69}, and interval thickness~\cite{GW95}, just to name a few.

\subparagraph*{The union \texorpdfstring{$\bm{\calG}$}{G}-covering number of $\bm{H}$} is
\[
 \cn{u}{\calG}{H} = \min\{t \in \N \setmid \text{there exists a $t$-global injective $\union{\calG}$-cover of $H$}\},  
\]
where $\union{\calG}$ denotes the class of all vertex-disjoint unions of graphs in $\calG$.
Hence, for the union $\calG$-covering number we cover $H$ with as few subgraphs as possible, each being a vertex-disjoint union of graphs in $\calG$.
For example, $\union{\{K_2\}}$ is the class of all matchings and hence $\cn{u}{\{K_2\}}{H}$ is the chromatic index of $H$~\cite{Koenig16}.
For another example, if $\calS = \{K_{1,n} \setmid n \in \N_{\geq 1}\}$ is the class of all stars, then $\cn{u}{\calS}{H}$ is known as the star arboricity of $H$~\cite{AK84}.

\subparagraph*{The local \texorpdfstring{$\bm{\calG}$}{G}-covering number of $\bm{H}$} is
\[
 \cn{l}{\calG}{H} = \min\{s \in \N \setmid \text{there exists an $s$-local injective $\calG$-cover of $H$}\}.  
\]
Here, we do not restrict the number of guests in the cover, but rather limit the number of guests we ``see locally'' at any vertex.
For instance, $\cn{l}{\{K_2\}}{H}$ is simply the maximum degree of~$H$.
Other local $\calG$-covering numbers from the literature include the bipartite degree~\cite{FH96}, local boxicity~\cite{blasius2018local_boxicity}, local Ramsey numbers~\cite{GLST87}, and local interval numbers~\cite{Knauer2016_3w3c1g}.

\subparagraph*{The folded \texorpdfstring{$\bm{\calG}$}{G}-covering number of $\bm{H}$} is
\[
 \cn{f}{\calG}{H} = \min\{s \in \N \setmid \text{there exists an $s$-local $\calG$-cover of $H$}\}.  
\]
Here, we no longer require each $\varphi(G_i)$ in a $\calG$-cover $\varphi$ to be a subgraph of $H$, and instead allow arbitrary homomorphic images of graphs in $\calG$.
Intuitively, we can ``fold'' (by identifying non-adjacent vertices) graphs from $\calG$ and then use these to cover all edges of $H$.
Alternatively, one can think of an $s$-local $\calG$-cover of $H$ as splitting each $v \in V(H)$ into at most $s$ vertices, distributing the incident edges among these (possibly with repetition), such that the resulting graph is a vertex-disjoint union $G_1 \cupdot \cdots \cupdot G_t$ of graphs in $\calG$.
For example, if $\calP$ denotes the class of all planar graphs, then $\cn{f}{\calP}{H}$ is the splitting number of $H$~\cite{HJR85}.
And if $\calI$ denotes the class of all interval graphs, then $\cn{f}{\calI}{H}$ is the interval number of $H$~\cite{TH79}.

\medskip

Finally, let us refer to \cref{fig:example} for illustrating examples of all four $\calG$-covering numbers.

\begin{figure}[ht]
    \centering
    \begin{subfigure}[t]{0.37\textwidth}
       \centering
        \includegraphics[page=5]{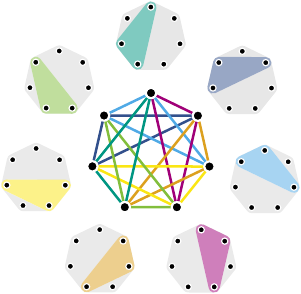}
        \caption{}
        \label{fig:examples_left}
    \end{subfigure}\hfill
    \begin{subfigure}[t]{0.37\textwidth}
        \centering
        \includegraphics[page=6]{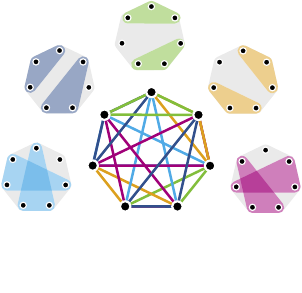}
        \caption{}
        \label{fig:examples_middle}
    \end{subfigure}\hfill
    \begin{subfigure}[t]{0.21\textwidth}
       \centering
        \includegraphics[page=7]{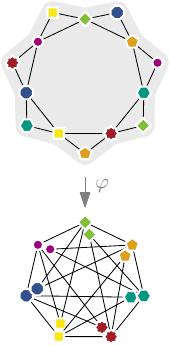}
        \caption{}
        \label{fig:examples_right}
    \end{subfigure}
    \caption{
        Examples of $\calG$-covers of $H$:
        For $\calG = \{K_3\}$ and $H = K_7$, we have a $7$-global, $3$-local injective $\calG$-cover (\subref{fig:examples_left}), as well as a $5$-global, $5$-local $\union{\calG}$-cover (\subref{fig:examples_middle}).
        In fact, it holds that $\cn{g}{\{K_3\}}{K_7} = 7$ and $\cn{u}{\{K_3\}}{K_7} = 5$ and $\cn{l}{\{K_3\}}{K_7} = 3$.
        (\subref{fig:examples_right}) A $2$-local non-injective $\calG$-cover of $K_7$ for $\calG = {\rm Forb}(C_4)$ being the class of $C_4$-free graphs.
        In fact, it holds that $\cn{f}{{\rm Forb}(C_4)}{K_7} = 2$, while $\cn{l}{{\rm Forb}(C_4)}{K_7} \geq 3$ since the local Ramsey number of $C_4$ is $6$~\cite{GLST87}.
    }
    \label{fig:example}
\end{figure}

Let us briefly mention that for some guest classes $\calG$ some host graph $H$ might admit no $\calG$-cover, or at least no injective $\calG$-cover.
For instance consider a tree $H$ as the host and the guest class $\calC_{\rm odd} = \{C_{2n+1} \setmid n \geq 1\}$ of all odd cycles.
Then $H$ has no $\calC_{\rm odd}$-cover and consequently $\cn{g}{\calC_{\rm odd}}{H} = \cn{u}{\calC_{\rm odd}}{H} = \cn{l}{\calC_{\rm odd}}{H} = \cn{f}{\calC_{\rm odd}}{H} = \infty$.
In turn, for the guest class $\calC_{\rm even} = \{C_{2n} \setmid n \geq 2\}$ of all even cycles, tree $H$ has only non-injective $\calC_{\rm even}$-covers and we have $\cn{g}{\calC_{\rm even}}{H} = \cn{u}{\calC_{\rm even}}{H} = \cn{l}{\calC_{\rm even}}{H} = \infty$ while $\cn{f}{\calC_{\rm even}}{H}$ is easily seen to be the maximum degree of $H$, provided $H \neq K_2$.
To prevent such pathologic cases, we shall assume\footnote{This holds in particular for every non-trivial hereditary guest class.} throughout this work that all guest classes~$\calG$ contain the graph~$K_2$.
With this assumption, every host graph~$H$ admits some (injective) $\calG$-cover.

\subsection{Our Results}
\label{subsec:our-results}

Recall the chain of inequalities in~\eqref{eq:gulf-chain}, namely that for every host graph~$H$ and every guest class~$\calG$ we have $\cn{g}{\calG}{H} \geq \cn{u}{\calG}{H} \geq \cn{l}{\calG}{H} \geq \cn{f}{\calG}{H}$.
Also recall that each of these three inequalities can be arbitrarily far apart.
In this paper, we present conditions for a guest class $\calG$ and a host class $\calH$, such that within graphs in $\calH$ a larger $\calG$-covering number can be bounded in terms of a smaller $\calG$-covering number.

\begin{definition}\label{def:functionally-bounded}
    For a guest class~$\calG$ and two $\calG$-covering numbers~$\cn{x}{\calG}{}$ and~$\cn{y}{\calG}{}$, we say that a class $\calH$ of host graphs is \emph{$(\cn{x}{\calG}{}, \cn{y}{\calG}{})$-bounded} if there exists a function~$f \colon \N \to \R_{\geq 0}$ such that 
    \[
        \cn{y}{\calG}{H} \leq \cn{x}{\calG}{H} \leq f( \cn{y}{\calG}{H}) \qquad \text{ for all graphs $H \in \calH$.}
    \]
    The function $f$ is then called a \emph{binding function}.
\end{definition}

It is most interesting to consider $(\cn{x}{\calG}{}, \cn{y}{\calG}{})$-boundedness for pairs $\cn{x}{\calG}{},\cn{y}{\calG}{}$ of $\calG$-covering numbers that appear consecutively in the inequality chain~\eqref{eq:gulf-chain}.
That is, we investigate under which conditions on $\calH$ and $\calG$ the class $\calH$ of host graphs is $(\cn{g}{\calG}{},\cn{u}{\calG}{})$-bounded or $(\cn{u}{\calG}{},\cn{l}{\calG}{})$-bounded or $(\cn{l}{\calG}{},\cn{f}{\calG}{})$-bounded.

We start with the case of $(\cn{g}{\calG}{},\cn{u}{\calG}{})$-boundedness.
Recall that for a graph class $\calG$, we denote by $\union{\calG}$ the class of all vertex-disjoint unions of graphs in $\calG$.
We say that $\calG$ is \emph{union-closed} if $\union{\calG} = \calG$.
Further, we call a graph class $\calG$ \emph{monotone} (also known as \emph{subgraph-closed}) if every subgraph of every graph in $\calG$ is again in $\calG$.
For example, the class of all planar graphs is union-closed and monotone, the class of all stars is neither union-closed nor monotone (removing an edge from a star results in a disconnected graph), while the class of all interval graphs is union-closed but not monotone.

\begin{restatable}[global vs.\ union]{theorem}{thmglobalunion}
    \label{thm:global-union-introduction}{\ \\}
    Let $\calG$ and $\calH$ be two graph classes.
    Then each of the following holds.
    \begin{enumerate}
        \item If $\calG$ is union-closed, then (clearly) for all graphs $H \in \calH$ we have $\cn{g}{\calG}{H} = \cn{u}{\calG}{H}$.
            In particular, $\calH$ is $(\cn{g}{\calG}{},\cn{u}{\calG}{})$-bounded.
        \item If $\calH$ is monotone, then
        \[
            \text{$\calH$ is $(\cn{g}{\calG}{},\cn{u}{\calG}{})$-bounded \quad if and only if \quad $\max\{ \cn{g}{\calG}{J} \setmid J \in \union{\calG}\cap \calH \} < \infty$.}
        \]
    \end{enumerate}
\end{restatable}

Also for the cases of $(\cn{u}{\calG}{},\cn{l}{\calG}{})$-boundedness and $(\cn{l}{\calG}{},\cn{f}{\calG}{})$-boundedness, we shall consider structural restrictions on $\calG$ and $\calH$, similar to being union-closed or monotone as in \cref{thm:global-union-introduction}.

\begin{remark*}
    Clearly, if a host class $\calH$ is $(\cn{x}{\calG}{}, \cn{y}{\calG}{})$-bounded, then every subclass $\calH' \subseteq \calH$ is $(\cn{x}{\calG}{}, \cn{y}{\calG}{})$-bounded as well.
    So, for a fixed guest class~$\calG$, the more we restrict the host class~$\calH$, the more likely it is that $\calH$ is $(\cn{x}{\calG}{}, \cn{y}{\calG}{})$-bounded.
    On the other hand, for a fixed host class~$\calH$ and two guest classes $\calG' \subseteq \calG$, there is no immediate relation between $\calH$ being $(\cn{x}{\calG}{}, \cn{y}{\calG}{})$-bounded and $\calH$ being $(\cn{x}{\calG'}{}, \cn{y}{\calG'}{})$-bounded.
    Having fewer guests in $\calG'$ than in $\calG$, we surely have $\cn{x}{\calG'}{H} \geq \cn{x}{\calG}{H}$ and $\cn{y}{\calG'}{H} \geq \cn{y}{\calG}{H}$ for every graph $H$.
    But the effects on the inequalities $\cn{x}{\calG'}{H} \leq f( \cn{y}{\calG'}{H})$ and $\cn{x}{\calG}{H} \leq f( \cn{y}{\calG}{H})$ are unpredictable if we only know that $\calG' \subseteq \calG$.\lipicsEnd
\end{remark*}

Specifically, we consider the following structural restrictions\footnote{Definitions are given below.} on $\calG$ and $\calH$.

\medskip

\noindent
\begin{minipage}[t]{0.52\textwidth}
    For the \underline{\smash{guest class $\calG$}} we consider:
    \smallskip
    \begin{enumerate}[(1)]
        \item $\calG$ is any graph class
        \item $\calG$ is component-closed
        \item $\calG$ is hereditary and \label{item:G_her}
        \begin{enumerate}[(i)]
            \item has no further restrictions
            \item has bounded chromatic number
            \item\label{item:G_her_mad} has bounded maximum average degree
        \end{enumerate}
        \item $\calG$ is monotone and\label{item:G_mon}
        \begin{enumerate}[(i)]
            \item\label{item:G_mon_none} has no further restrictions
            \item\label{item:G_mon_chi} has bounded chromatic number
            \item\label{item:G_mon_mad} has bounded maximum average degree
        \end{enumerate}
    \end{enumerate}
\end{minipage}
\hfill
\begin{minipage}[t]{0.43\textwidth}
    \vspace{-3pt}
    
    \includegraphics[page=1]{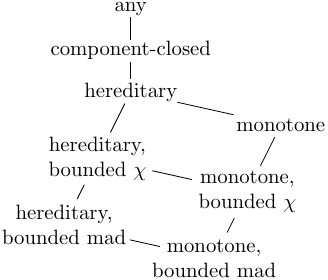}
\end{minipage}

\vspace{2em}

\noindent
\begin{minipage}[t]{0.6\textwidth}
    For the \underline{host class $\calH$} we consider:
    \smallskip
    \begin{enumerate}[(1)]
        \item $\calH$ is any graph class
        \item $\calH$ has bounded chromatic number
        \item $\calH$ has bounded maximum average degree
        \item $\calH$ is $M$-minor-free
        \item $\calH$ has bounded treewidth
    \end{enumerate}
\end{minipage}
\hfill
\begin{minipage}[t]{0.39\textwidth}
    \vspace{-15pt}
    
    \includegraphics[page=2]{restrictions.pdf}
\end{minipage}

\vspace{1.5em}

Recall that a graph class $\calG$ is monotone if for every graph $G \in \calG$ also every subgraph of $G$ is in $\calG$.
A subgraph~$G'$ of a graph~$G$ is an \emph{induced} subgraph of~$G$ if for each pair~$u,v \in V(G')$ with~$uv\in E(G)$, we have $uv\in E(G')$.
Then, $\calG$ is \emph{hereditary} if for every graph $G \in \calG$ also every \emph{induced} subgraph of $G$ is in $\calG$.
Further, $\calG$ is \emph{component-closed} if for every graph $G \in \calG$ also every connected component of $G$ is in $\calG$.
Finally for a graph $M$, $\calG$ is \emph{$M$-minor-free} if no graph in $\calG$ contains $M$ as a minor.

For a graph $G$ we denote its \emph{chromatic number} by $\chi(G)$, its \emph{treewidth} by $\tw(G)$, and its \emph{maximum average degree} (mad for short) by 
\[
    \mad(G) = \max\left\{ \frac{2\abs{E(G')}}{\abs{V(G')}} \setmid G' \subseteq G, \abs{V(G')} > 0 \right\}.
\]
For a graph class $\calG$ and a graph parameter $p \colon \calG \to \R_{\geq 0}$ we put $p(\calG) = \sup\{ p(G) \setmid G \in \calG\}$.
We then say that $\calG$ has \emph{bounded $p$} if $p(\calG) < \infty$.
For example, if $\calG$ has bounded chromatic number, then there exists a constant $C$ such that for all graphs $G \in \calG$ we have $\chi(G) \leq C$, i.e., all graphs in $\calG$ are $C$-colorable.
With this (standard) terminology, observe that a class $\calH$ is $(\cn{x}{\calG}{},\cn{y}{\calG}{})$-bounded if and only if every subclass $\calH' \subseteq \calH$ with bounded $\cn{y}{\calG}{}$ also has bounded $\cn{x}{\calG}{}$.

In our lists of structural restrictions for $\calG$ and $\calH$ above, some properties are more restrictive than others.
For example, \eqref{item:G_mon} implies \eqref{item:G_her}, \eqref{item:G_mon_mad} implies \eqref{item:G_mon_chi}, while on the other hand \eqref{item:G_mon_none} does not imply \eqref{item:G_her_mad}.
The poset of all relations is given next to the respective list.
Roughly speaking, we consider gradually more restricted graph classes as we go down in the lists.
Hence, as we go down in the lists, it becomes more likely that $\calH$ is $(\cn{u}{\calG}{},\cn{l}{\calG}{})$-bounded or $(\cn{l}{\calG}{},\cn{f}{\calG}{})$-bounded.
We prove the following results, which are summarized in \cref{fig:overview}.

\begin{table}[ht]
    \ra{1.3}
    \def\smallDist{3}
    \def\largeDist{8}
    \small
    \centering
    \begin{subtable}[t]{\textwidth}
        \centering
        \begin{tabularx}{0.83\textwidth}{c c@{\extracolsep{0pt}} @{\extracolsep{\smallDist pt}}c@{\extracolsep{0cm}} @{\extracolsep{\largeDist pt}}c@{\extracolsep{0cm}} @{\extracolsep{\smallDist pt}}c@{\extracolsep{0cm}} @{\extracolsep{\smallDist pt}}c@{\extracolsep{0cm}} @{\extracolsep{\largeDist pt}}c@{\extracolsep{0cm}} @{\extracolsep{\smallDist pt}}c@{\extracolsep{0cm}} @{\extracolsep{\smallDist pt}}c@{\extracolsep{0cm}}}
            \toprule
            \multirow{2}{*}{\diagbox[width=6em]{\phantom{vu}$\calH$}{$\calG$}} & \multirow{2}{*}{any} & \multirow{2}{*}{cc} & \multicolumn{3}{c}{hereditary} & \multicolumn{3}{c}{monotone} \\
            \arrayrulecolor{black}\cmidrule(r){4-6} \cmidrule(l){7-9}
            &   &    & \multicolumn{1}{c}{any} & $\chi$ & $\mad$ & \multicolumn{1}{c}{any} & $\chi$ & \multicolumn{1}{c}{$\mad$} \\
            \cmidrule{2-9}
            \arrayrulecolor{white}
            any & \multicolumn{1}{|c|}{\cellcolor{red30}\phantom{\shortref{lem:sep_folded-local_G_cc_H_tw}}} & \multicolumn{1}{|c|}{\cellcolor{red30}\phantom{\shortref{lem:sep_folded-local_G_cc_H_tw}}} & \multicolumn{1}{|c|}{\cellcolor{red30}} & \multicolumn{1}{|c|}{\cellcolor{red30}\phantom{\shortref{thm:union_bounded_by_folded_G_her_mad}}} & \multicolumn{1}{|c|}{\cellcolor{green30}\shortref{thm:union_bounded_by_folded_G_her_mad}} & \multicolumn{1}{|c|}{\cellcolor{red30}} & \multicolumn{1}{|c|}{\cellcolor{red30}\shortref{lem:sep_local_union_H_mon_G_mon_chi}} & \multicolumn{1}{|c|}{\cellcolor{green30}\phantom{\shortref{thm:union_bounded_by_folded_G_her_mad}}} \\
            \cmidrule{2-9}
            bounded $\chi$ & \multicolumn{1}{|c|}{\cellcolor{red30}} & \multicolumn{1}{|c|}{\cellcolor{red30}} & \multicolumn{1}{|c|}{\cellcolor{green30}\shortref{thm:union_bounded_by_folded_G_hereditary_H_chi}} & \multicolumn{1}{|c|}{\cellcolor{green30}} & \multicolumn{1}{|c}{\cellcolor{green30}} & \multicolumn{1}{|c|}{\cellcolor{green30}\smallshortref{thm:union_bounded_by_folded_G_monotone_H_chi}} & \multicolumn{1}{|c|}{\cellcolor{green30}} & \multicolumn{1}{|c|}{\cellcolor{green30}}\\
             \cmidrule{2-9}
            bounded $\mad$ & \multicolumn{1}{|c|}{\cellcolor{red30}} & \multicolumn{1}{|c}{\cellcolor{red30}} & \multicolumn{1}{|c|}{\cellcolor{green30}} & \multicolumn{1}{|c|}{\cellcolor{green30}} & \multicolumn{1}{|c}{\cellcolor{green30}} & \multicolumn{1}{|c|}{\cellcolor{green30}} & \multicolumn{1}{|c|}{\cellcolor{green30}} & \multicolumn{1}{|c|}{\cellcolor{green30}}\\
             \cmidrule{2-9}
            $M$-minor-free & \multicolumn{1}{|c|}{\cellcolor{red30}} & \multicolumn{1}{|c}{\cellcolor{red30}\shortref{thm:minor_comp_sep}} & \multicolumn{1}{|c|}{\cellcolor{green30}} & \multicolumn{1}{|c|}{\cellcolor{green30}} & \multicolumn{1}{|c}{\cellcolor{green30}} & \multicolumn{1}{|c|}{\cellcolor{green30}} & \multicolumn{1}{|c|}{\cellcolor{green30}} & \multicolumn{1}{|c|}{\cellcolor{green30}}\\
             \cmidrule{2-9}
            bounded $\tw$ & \multicolumn{1}{|c|}{\cellcolor{red30}\shortref{thm:tw_sep}} & \multicolumn{1}{|c|}{\cellcolor{green30}\shortref{thm:cc_tw_bound}}& \multicolumn{1}{|c|}{\cellcolor{green30}} & \multicolumn{1}{|c|}{\cellcolor{green30}} & \multicolumn{1}{|c|}{\cellcolor{green30}} & \multicolumn{1}{|c|}{\cellcolor{green30}} & \multicolumn{1}{|c|}{\cellcolor{green30}} & \multicolumn{1}{|c|}{\cellcolor{green30}}\\
            \arrayrulecolor{black}\bottomrule
        \end{tabularx}
        
        \medskip
        
        \caption{Is $\calH$ $(\cn{u}{\calG}{},\cn{l}{\calG}{})$-bounded? \;\; \textcolor{green30}{$\blacksquare$} Always Yes \;\; \textcolor{red30}{$\blacksquare$} Sometimes No}
        \label{fig:union-local}
    \end{subtable}
    
    \begin{subtable}[t]{\textwidth}
        \centering
        \begin{tabularx}{0.83\textwidth}{c c@{\extracolsep{0pt}} @{\extracolsep{\smallDist pt}}c@{\extracolsep{0cm}} @{\extracolsep{\largeDist pt}}c@{\extracolsep{0cm}} @{\extracolsep{\smallDist pt}}c@{\extracolsep{0cm}} @{\extracolsep{\smallDist pt}}c@{\extracolsep{0cm}} @{\extracolsep{\largeDist pt}}c@{\extracolsep{0cm}} @{\extracolsep{\smallDist pt}}c@{\extracolsep{0cm}} @{\extracolsep{\smallDist pt}}c@{\extracolsep{0cm}}}
            \toprule
            \multirow{2}{*}{\diagbox[width=6em]{\phantom{vu}$\calH$}{$\calG$}} & \multirow{2}{*}{any} & \multirow{2}{*}{cc} & \multicolumn{3}{c}{hereditary} & \multicolumn{3}{c}{monotone} \\
            \arrayrulecolor{black}\cmidrule(r){4-6} \cmidrule(l){7-9}
            &   &    & \multicolumn{1}{c}{any} & $\chi$ & $\mad$ & \multicolumn{1}{c}{any} & $\chi$ & \multicolumn{1}{c}{$\mad$} \\
            \cmidrule{2-9}
            \arrayrulecolor{white}
            any & \multicolumn{1}{|c|}{\cellcolor{red30}\phantom{\shortref{lem:sep_folded-local_G_cc_H_tw}}} & \multicolumn{1}{|c|}{\cellcolor{red30}} & \multicolumn{1}{|c|}{\cellcolor{red30}} & \multicolumn{1}{|c|}{\cellcolor{red30}\phantom{\shortref{thm:union_bounded_by_folded_G_her_mad}}} & \multicolumn{1}{|c|}{\cellcolor{green30}\shortref{thm:union_bounded_by_folded_G_her_mad}} & \multicolumn{1}{|c|}{\cellcolor{red30}} &  \multicolumn{1}{|c|}{\cellcolor{red30}\shortref{lem:sep_folded-folded_G_mon_chi_H_mon}} & \multicolumn{1}{|c|}{\cellcolor{green30}\phantom{\shortref{thm:union_bounded_by_folded_G_her_mad}}} \\
            \cmidrule{2-9}
            bounded $\chi$ & \multicolumn{1}{|c|}{\cellcolor{red30}} & \multicolumn{1}{|c}{\cellcolor{red30}} & \multicolumn{1}{|c|}{\cellcolor{green30}\shortref{thm:union_bounded_by_folded_G_hereditary_H_chi}} & \multicolumn{1}{|c|}{\cellcolor{green30}} & \multicolumn{1}{|c}{\cellcolor{green30}} & \multicolumn{1}{|c|}{\cellcolor{green30}\smallshortref{thm:union_bounded_by_folded_G_monotone_H_chi}} & \multicolumn{1}{|c|}{\cellcolor{green30}} & \multicolumn{1}{|c|}{\cellcolor{green30}}\\
             \cmidrule{2-9}
            bounded $\mad$ & \multicolumn{1}{|c|}{\cellcolor{red30}} & \multicolumn{1}{|c}{\cellcolor{red30}} & \multicolumn{1}{|c|}{\cellcolor{green30}} & \multicolumn{1}{|c|}{\cellcolor{green30}} & \multicolumn{1}{|c}{\cellcolor{green30}} & \multicolumn{1}{|c|}{\cellcolor{green30}} & \multicolumn{1}{|c|}{\cellcolor{green30}} & \multicolumn{1}{|c|}{\cellcolor{green30}}\\
             \cmidrule{2-9}
            $M$-minor-free & \multicolumn{1}{|c}{\cellcolor{red30}} & \multicolumn{1}{|c}{\cellcolor{red30}} & \multicolumn{1}{|c|}{\cellcolor{green30}} & \multicolumn{1}{|c|}{\cellcolor{green30}} & \multicolumn{1}{|c}{\cellcolor{green30}} & \multicolumn{1}{|c|}{\cellcolor{green30}} & \multicolumn{1}{|c|}{\cellcolor{green30}} & \multicolumn{1}{|c|}{\cellcolor{green30}}\\
             \cmidrule{2-9}
            bounded $\tw$ & \multicolumn{1}{|c|}{\cellcolor{red30}} & \multicolumn{1}{|c|}{\cellcolor{red30}\shortref{lem:sep_folded-local_G_cc_H_tw}}& \multicolumn{1}{|c|}{\cellcolor{green30}} & \multicolumn{1}{|c|}{\cellcolor{green30}} & \multicolumn{1}{|c|}{\cellcolor{green30}} & \multicolumn{1}{|c|}{\cellcolor{green30}} & \multicolumn{1}{|c|}{\cellcolor{green30}} & \multicolumn{1}{|c|}{\cellcolor{green30}}\\
            \arrayrulecolor{black}\bottomrule
        \end{tabularx}
        
        \medskip
        
        \caption{Is $\calH$ $(\cn{l}{\calG}{},\cn{f}{\calG}{})$-bounded? \;\; \textcolor{green30}{$\blacksquare$} Always Yes \;\; \textcolor{red30}{$\blacksquare$} Sometimes No}
        \label{fig:local-folded}
    \end{subtable}
    \caption{
        Overview of the results in \cref{thm:union-local-introduction,thm:local-folded-introduction}.
        \cref{fig:union-local} shows whether every host class~$\calH$ with a specific restriction (any, bounded $\chi$, bounded $\mad$, $M$-minor-free, bounded $\tw$) is $(\cn{u}{\calG}{},\cn{l}{\calG}{})$-bounded for every guest class $\calG$ with a specific restriction (any, component-closed, hereditary, monotone, bounded $\chi$, bounded $\mad$).
        Numbers $\langle X \rangle$ refer to the corresponding Theorem~$X$ in the paper, while all unnumbered results are immediate consequences.
        \cref{fig:local-folded} shows the analogous results for $(\cn{l}{\calG}{},\cn{f}{\calG}{})$-boundedness.
    }
    \label{fig:overview}
\end{table}

\begin{theorem}[union vs.\ local]
    \label{thm:union-local-introduction}{\ \\}
    For any two graph classes $\calG,\calH$ we have that $\calH$ is $(\cn{u}{\calG}{},\cn{l}{\calG}{})$-bounded, provided one of the following holds:
    \begin{enumerate}
        \item $\calG$ is hereditary and has bounded maximum average degree.
        \item $\calG$ is hereditary and $\calH$ has bounded chromatic number.
        \item $\calG$ is component-closed and $\calH$ has bounded treewidth.
    \end{enumerate}
    There exist two graph classes $\calG,\calH$ such that $\calH$ is \emph{not} $(\cn{u}{\calG}{},\cn{l}{\calG}{})$-bounded in each of the following cases:
    \begin{enumerate}
        \item $\calG$ is monotone and has bounded chromatic number.
        \item $\calG$ is component-closed and $\calH$ is $M$-minor-free.
        \item $\calH$ has bounded treewidth.
    \end{enumerate}    
\end{theorem}

\begin{theorem}[local vs.\ folded]
    \label{thm:local-folded-introduction}{\ \\}
    For any two graph classes $\calG,\calH$ we have that $\calH$ is $(\cn{l}{\calG}{},\cn{f}{\calG}{})$-bounded, provided one of the following holds:
    \begin{enumerate}
        \item $\calG$ is hereditary and has bounded maximum average degree.
        \item $\calG$ is hereditary and~$\calH$ has bounded chromatic number.
    \end{enumerate}
    There exist two graph classes $\calG,\calH$ such that $\calH$ is \emph{not} $(\cn{l}{\calG}{},\cn{f}{\calG}{})$-bounded in each of the following cases:
    \begin{enumerate}
        \item $\calG$ is monotone and has bounded chromatic number.\label{item:local-folded-3}
        \item $\calG$ is component-closed and $\calH$ has bounded treewidth.\label{item:local-folded-4}
    \end{enumerate}    
\end{theorem}

\subparagraph*{Organization of the paper.}

We start by proving \cref{thm:global-union-introduction} in \cref{sec:global-union-separation}, which fully treats $(\cn{g}{\calG}{},\cn{u}{\calG}{})$-boundedness.
The subsequent sections are then devoted to $(\cn{u}{\calG}{},\cn{l}{\calG}{})$-boundedness and $(\cn{l}{\calG}{},\cn{f}{\calG}{})$-boundedness.
We consider hereditary and non-hereditary guest classes separately.

In \cref{sec:non-hereditary-guest-classes} we consider guest classes $\calG$ that are not necessarily hereditary.
We prove all results corresponding to the first two columns of \cref{fig:union-local,fig:local-folded} here.
We find that whether or not a host class $\calH$ is $(\cn{u}{\calG}{},\cn{l}{\calG}{})$-bounded (similarly for $(\cn{l}{\calG}{},\cn{f}{\calG}{})$-boundedness) mostly depends on whether $\calH$ has bounded treewidth and/or $\calG$ is component-closed.

In \cref{sec:hereditary-guest-classes} we then consider hereditary guest classes~$\calG$.
We start in \cref{subsec:hereditary-sparse-guests} with the case that $\calG$ has additionally bounded $\mad$, in which case $\calG$ is also called \emph{sparse}.
In this case, every host class~$\calH$ turns out to be $(\cn{u}{\calG}{},\cn{f}{\calG}{})$-bounded and hence $(\cn{u}{\calG}{},\cn{l}{\calG}{})$-bounded as well as $(\cn{l}{\calG}{},\cn{f}{\calG}{})$-bounded; see \cref{thm:union_bounded_by_folded_G_her_mad}.
Similarly, if $\calG$ is hereditary but not necessarily sparse, then every host class~$\calH$ of bounded chromatic number is $(\cn{u}{\calG}{},\cn{f}{\calG}{})$-bounded; see \cref{thm:union_bounded_by_folded_G_hereditary_H_chi} in \cref{subsec:hereditary-guests-chromatic-hosts}.
The last case is that $\calG$ is hereditary but $\calH$ has unbounded chromatic number.
This is the subject of \cref{subsec:G-her-H-unbounded-X}.
Here we can (for some guest and host classes) separate the union and local $\calG$-covering number (\cref{subsec:G-her-union-local-separation}), as well as the local and folded $\calG$-covering number (\cref{subsec:G-her-local-folded-separation}).
This means that we provide a particular hereditary guest class $\calG$ and a particular host class $\calH$ that is not $(\cn{u}{\calG}{},\cn{l}{\calG}{})$-bounded, respectively not $(\cn{l}{\calG}{},\cn{f}{\calG}{})$-bounded.

Finally, in \cref{sec:hereditary_hosts} we discuss a similar, yet different, concept to $(\cn{x}{\calG}{},\cn{y}{\calG}{})$-boundedness.
Namely, we ask whether (for a given guest class~$\calG$) a host class $\calH$ with bounded $\cn{y}{\calG}{}$ also has bounded $\cn{x}{\calG}{}$.
If yes, this implies that $\calH$ is $(\cn{x}{\calG}{},\cn{y}{\calG}{})$-bounded, but the converse is not necessarily true.
For example, we show that if $\calG$ is component-closed and $\calH$ is hereditary with bounded $\mad$, then $\cn{l}{\calG}{\calH} < \infty$ implies $\cn{u}{\calG}{\calH} < \infty$.
(Under the same conditions, $\calH$ is not always $(\cn{u}{\calG}{},\cn{l}{\calG}{})$-bounded; cf.~\cref{thm:minor_comp_sep}.)

We conclude the paper with a discussion of open problems in \cref{sec:conclusions}.

\section{Global and Union Covering Number}
\label{sec:global-union-separation}

The global $\calG$-covering number $\cn{g}{\calG}{H}$ of a graph $H$ is the minimum $t$ such that $H$ can be covered with (equivalently, is the union of) $t$ graphs from $\calG$.
For the union $\calG$-covering number $\cn{u}{\calG}{H}$ we are allowed to cover $H$ with $t$ vertex-disjoint unions of graphs from $\calG$.
Hence, whenever the guest class $\calG$ is union-closed, the global and union $\calG$-covering number of every graph $H$ coincide.
Most $\calG$-covering numbers that are studied in the literature are based on union-closed guest classes $\calG$ (like forests, outerplanar or planar graphs, interval graphs, etc.).
For such guest classes~$\calG$, every host class~$\calH$ is $(\cn{g}{\calG}{},\cn{u}{\calG}{})$-bounded, which is the first statement of \cref{thm:global-union-introduction}.

Yet, other guest classes~$\calG$ (like the class of co-interval graphs, or all $T$-free graphs for a disconnected $T$) are not union-closed, and hence make the consideration of $(\cn{g}{\calG}{},\cn{u}{\calG}{})$-boundedness of host classes $\calH$ non-trivial and interesting.
Here, we give a characterization for the case that $\calH$ is monotone.

\thmglobalunion*
\begin{proof}
    To prove the second statement, first assume that~$\calH$ is $(\cn{g}{\calG}{},\cn{u}{\calG}{})$-bounded, i.e. there exists a function~$f\colon \N \to \N$ such that for every~$H \in \calH$, we have $\cn{g}{\calG}{H} \leq f(\cn{u}{\calG}{H})$.
    In particular, every graph $J \in \union{\calG} \cap \calH$ satisfies $\cn{g}{\calG}{J} \leq f(\cn{u}{\calG}{J}) \leq f(1)$.
    That is, $\cn{g}{\calG}{\union{\calG} \cap \calH} < \infty$.

    Now suppose there exists a constant~$s$ such that $\cn{g}{\calG}{\union{\calG} \cap \calH} \leq s$. 
    We show that $\cn{g}{\calG}{H} \leq s \cn{u}{\calG}{H}$ for every graph~$H \in \calH$.
    It then follows that~$\calH$ is  $(\cn{g}{\calG}{},\cn{u}{\calG}{})$-bounded.
    
    Let~$H \in \calH$ and consider an injective $t$-global $\union{\calG}$-cover~$\varphi'\colon J_1 \cupdot \dots \cupdot J_t \to H$ with $t = \cn{u}{\calG}{H}$.
    Each~$J_i$ is in particular a subgraph of~$H$ and as~$\calH$ is monotone, $J_i \in \union{\calG} \cap \calH$ follows.
    Since $\cn{g}{\calG}{\union{\calG} \cap \calH} \leq s$, for each graph~$J_i$ there exists an injective $s$-global $\calG$-cover $\varphi_i \colon G_{i,1} \cupdot \dots \cupdot G_{i,s} \to J_i$.
    We obtain a homomorphism $\varphi \colon G_{1,1} \cupdot \dots \cupdot G_{t,s} \to H$ by setting $\varphi|_{G_{i,j}}= \varphi_i|_{G_{i,j}}$ for each $i \in [t], j \in [s]$.
    Note that~$\varphi$ is an injective $st$-global $\calG$-cover of~$H$, which certifies $\cn{g}{\calG}{H} \leq s \cn{u}{\calG}{H}$.
\end{proof}

With \cref{thm:global-union-introduction} at hand, it is easy to find a guest class $\calG$ and a host class $\calH$ that is not $(\cn{g}{\calG}{},\cn{u}{\calG}{})$-bounded, even though $\calG$ and $\calH$ are structurally very well-behaved.
For example, consider~$\calG$ to be all graphs on at most two vertices (that is, $\calG = \{K_2,K_1,2K_1\}$) and~$\calH$ to be all graphs of maximum degree at most~$1$ (that is, $\calH = \{aK_2 + bK_1 \mid a,b \geq 0\}$).
Then $\calH$ is not $(\cn{g}{\calG}{},\cn{u}{\calG}{})$-bounded, even though~$\calG$ and~$\calH$ are both of bounded treewidth (hence bounded $\mad$) and monotone.

\section{Non-Hereditary Guest Classes}
\label{sec:non-hereditary-guest-classes}

Before we investigate hereditary guest classes in detail in \cref{sec:hereditary-guest-classes}, let us consider here guest classes~$\calG$ that are not necessarily hereditary.
In particular, we shall prove our results on $(\cn{u}{\calG}{},\cn{l}{\calG}{})$-boundedness and $(\cn{l}{\calG}{},\cn{f}{\calG}{})$-boundedness corresponding to the first two columns in \cref{fig:union-local,fig:local-folded}.
We start with a negative result.
If we put no additional assumptions on $\calG$, then even for very simple (i.e., structurally well-behaving) host classes~$\calH$ we are not guaranteed to have $(\cn{u}{\calG}{},\cn{l}{\calG}{})$-boundedness.

\begin{theorem}\label{thm:tw_sep}
    There is a host class $\calH$ with $\tw(\calH)= 1$ and a guest class $\calG$ such that $\calH$ is not $(\cn{u}{\calG}{},\cn{l}{\calG}{})$-bounded.
\end{theorem}
\begin{proof}
    For an integer $t \geq 4$ and $i\in[t]$ we define the graph $\hat{G_i^t}$ as a combination of a path $P_{2t+1} = p_0,p_1,\dots,p_{2t}$ on $2t+1$ vertices, two stars $K_{1,t}$ on $t+1$ vertices, and one $K_2$ (on two vertices).
    The central vertex of one $K_{1,t}$ is identified with $p_0$; the central vertex of the other $K_{1,t}$ with $p_{2t}$.
    One vertex of the $K_2$ is identified with the vertex $p_i$ on the path; see Figure~\ref{fig:twSep}.
    Further, let $G_i^t$ be the vertex-disjoint union of $t-1$ copies of $\hat{G_i^t}$.
    We define the guest class $\calG$ as $\calG = \{G_i^t\setmid t \geq 4, i\in [t]\}\cup\{K_2\}$.

    \begin{figure}[ht]
        \centering
        \includegraphics[page=2]{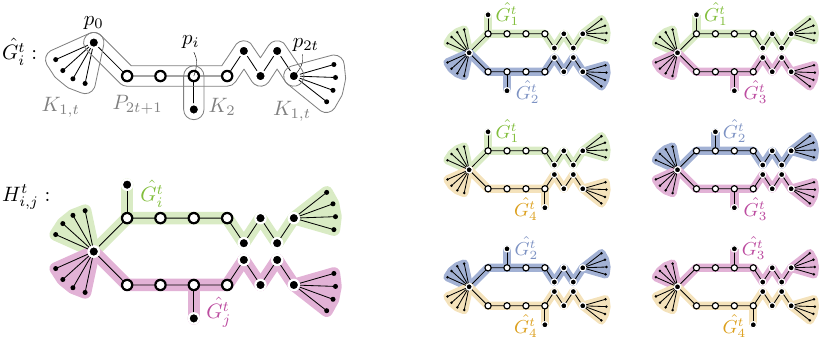}
        \caption{%
            Left: Illustrations of the graphs $\hat{G_i^t}$ for $i=3,t=4$, and $H_{i,j}^t$ for $i=1,j=3,t=4$.
            Right: A $2$-local $t$-global injective $\calG$-cover of $H_t$ for $t=4$.
        }
        \label{fig:twSep}
    \end{figure}

    Let $\calH$ be the class of all graphs with treewidth~$1$, that is, the class of all forests.
    For our argument, we consider graphs~$H_t$ with $t \geq 4$ in $\calH$ defined as follows.
    For an integer $t \geq 4$, let~$H_t$ be the vertex-disjoint union of the graphs $\{H_{i,j}^t\setmid 1 \leq i<j \leq t\}$ where $H_{i,j}^t$ is the graph obtained by identifying the corresponding vertices $p_0$ in $\hat{G_i^t}$ and in $\hat{G_j^t}$; see Figure~\ref{fig:twSep}. 
    Note that $\tw(\calH) = 1$.
    We shall show that $\cn{l}{\calG}{\calH} = 2$, while $\cn{u}{\calG}{\calH} = \infty$, and therefore that $\calH$ is not $(\cn{u}{\calG}{},\cn{l}{\calG}{})$-bounded.
    
    First observe that for every $t\geq 4$ host $H_t$ can be covered by $G_1^t,\ldots,G_t^t$ as follows.
    For each $i,j$ with $1 \leq i < j \leq t$, we cover the component of $H_t$ isomorphic to $H_{i,j}^t$ by a copy of $\hat{G_i^t}$ from $G_i^t$ and a copy of $\hat{G_j^t}$ from $G_j^t$ (corresponding to the definition of $H_{i,j}^t$).
    For each $i \in [t]$, this uses $t-1$ copies of $\hat{G_i^t}$ which can be combined to be one copy of $G_i^t$.
    See \cref{fig:twSep} for an illustration of this $2$-local injective $\calG$-cover of $H_t$.
    Note that each vertex of $H_t$ is indeed contained in at most two guests, i.e., $\cn{l}{\calG}{H_t} \leq 2$.

    To show that $\calH$ is not $(\cn{u}{\calG}{},\cn{l}{\calG}{})$-bounded, it now suffices to show that $\cn{u}{\calG}{H_t} \geq t$ for every $t \geq 4$.
    To this end, consider any injective $\union{\calG}$-cover $\varphi$ of $H_t$, i.e., a cover of $H_t$ with vertex-disjoint unions of graphs in $\calG$.
    Our goal is to show that $\varphi$ uses at least $t$ graphs from~$\union{\calG}$.
    Observe that for every $t' \neq t$ and every $i' \in [t']$ we have $G_{i'}^{t'} \not\subseteq H_t$.
    Hence, $\varphi$ uses only copies of $K_2$ and $G_1^t,\ldots,G_t^t$ to cover $H_t$.
    
    Consider any connected component $K$ of $H_t$ and let $v$ be one of the two vertices of degree~$t+1$ in $K$.
    If the $t$ pendant edges at $v$ are not covered by some $G_i^t$, then $\varphi$ uses at least $t$ different copies of $K_2$, all sharing vertex $v$ and we are done.
    Otherwise, $\varphi$ uses some $G_i^t$ to cover the pendant edges at $v$, and necessarily $i$ is (unique) number such that the nearest degree-$3$ vertex has distance $2t-i$ to $v$ in $K$.
    As this holds for every component $K$ of $H_t$ and every degree-$t+1$ vertex in $K$, we have that $\varphi$ uses each of $G_1^t,\ldots,G_t^t$.
    Crucially, for each $i\in [t]$ there is exactly one subgraph of $H_t$ that is isomorphic to $G_i^t$.
    And for every $j \neq i$, the copy of $G_i^t$ in $H_t$ shares a vertex with the copy of $G_j^t$ in $H_t$ (in the component of $H_t$ isomorphic to $H_{i,j}^t$).
    Hence, $\varphi$ uses $t$ pairwise intersecting graphs from $\calG$, which implies that $\cn{u}{\calG}{H_t} \geq t$, as desired.
\end{proof}

As mentioned above, \cref{thm:tw_sep} shows that we must put some restriction on the guest class $\calG$ in order to conclude that all host classes~$\calH$ of some kind are $(\cn{u}{\calG}{},\cn{l}{\calG}{})$-bounded.
In fact, our weakest restriction, namely $\calG$ being component-closed, is already enough.

\begin{theorem}\label{thm:cc_tw_bound}
    Let $\calG$ be a component-closed guest class and let $\calH$ be a host class with $\tw(\calH) = w <\infty$.
    Then we have $\cn{u}{\calG}{H}\leq (w+1) \cn{l}{\calG}{H}$ for all $H\in \calH$.
\end{theorem}
\begin{proof}
    Fix $H \in \calH$, let $s =\cn{l}{\calG}{H}$, and let $\varphi\colon G_1 \cupdot \cdots \cupdot G_t \to H$ be an $s$-local injective $\calG$-cover of $H$.
    Since $\calG$ is component-closed, we can assume without loss of generality that all guests $G_1,\dots,G_t$ used in $\varphi$ are connected.
    Consider a tree-decomposition of $H$ with host tree $T$ and bags of size of at most $w+1$.
    By definition of tree-decompositions, for each vertex $v$ of~$H$ there is a corresponding subtree $T_v$ of $T$ consisting of all vertices of $T$ whose bag contains $v$.
    Whenever $u$ and $v$ are adjacent in~$H$, their subtrees $T_u$ and $T_v$ share at least one vertex.
    
    For $i\in[t]$ let $T_i$ be the union of all $T_v$ for which $\varphi^{-1}(v)$ contains a vertex of $G_i$.
    Since $G_i$ is connected, $T_i$ is connected and thus a subtree of $T$, which we call the \emph{guest-subtree} of $G_i$.
    If two guest-subtrees $T_i,T_j$ are vertex-disjoint in~$T$, then the corresponding graphs $G_i,G_j$ are vertex-disjoint in~$H$.
    Our goal is to color the guest-subtrees with $(w+1)s$ colors such that the guest-subtrees in each color class are pairwise vertex-disjoint and hence the corresponding graphs in $H$ form together one graph in $\union{\calG}$.
    
    To this end, consider the intersection graph of all guest-subtrees, that is graph $S = ([t], \{\{i,j\}\setmid V(T_i)\cap V(T_j)\neq \emptyset\})$.    
    Since the family of trees has the Helly property, for each clique $X$ in $S$ there is a node $\mu$ of $T$ that is contained in the guest-subtree of $G_i$ for each $i \in X$.
    Since the bag of $\mu$ holds at most $w+1$ vertices of $H$ and each of these vertices is contained in at most $s$ guests, we conclude that $|X| \leq (w+1)s$.
    Since $S$ is the intersection graph of subtrees of a tree, it is a chordal graph~\cite{gavril1974intersection} and thus perfect~\cite{hajnal1958auflosung,Ber61}.
    Hence, $\chi(S) \leq (w+1)s$ and there is a proper $(w+1)s$-coloring of $S$.
    This means that the guests $G_1,\ldots,G_t$ can be partitioned into $(w+1)s$ sets such that the guests within each set are vertex-disjoint.
    Taking for each of those $(w+1)s$ sets its union as a new guest results in a $(w+1)s$-global injective $\union{\calG}$-cover of $H$.
    Thus $\cn{u}{\calG}{H} \leq (w+1)s = (w+1) \cn{l}{\calG}{H}$, which concludes the proof.
\end{proof}

\cref{thm:cc_tw_bound} states that if $\calG$ is any component-closed guest class, then a host class~$\calH$ is $(\cn{u}{\calG}{},\cn{l}{\calG}{})$-bounded, provided that graphs in $\calH$ have bounded treewidth.
Next we show that the bounded treewidth of $\calH$ is crucial for that statement.

\begin{theorem}\label{thm:minor_comp_sep}
    There is an $M$-minor-free host class $\calH$ and a component-closed guest class $\calG$ such that $\calH$ is not $(\cn{u}{\calG}{},\cn{l}{\calG}{})$-bounded.
\end{theorem}
\begin{proof}
    Let~$\calH$ be the class of all planar graphs, and note that $\calH$ is for instance $K_5$-minor-free.
    For our argument, we consider graphs $H_\ell$ with integer $\ell \geq 4$ in $\calH$ defined as follows; see Figure~\ref{fig:gridSep}.

    \begin{figure}[ht]
        \centering
        \includegraphics[page=1]{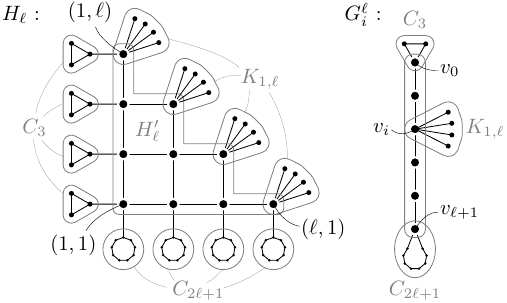}
        \hfill
        \includegraphics[page=2]{gridSep_v2}
        \caption{%
            Illustration of the graph $H_\ell$ for $\ell = 4$ (left) and $G_i^\ell$ for $i=2,\ell=4$ (middle).
            Right: a $2$-local $\ell$-global injective $\calG$-cover of $H_\ell$ for $\ell = 4$.
        }
        \label{fig:gridSep}
    \end{figure}
    
    For the construction of $H_\ell = (V_\ell,E_\ell)$, we start with the subgraph $H'_\ell$ of the $(\ell\times\ell)$-grid~$P_{\ell}\, \square \, P_{\ell}$ induced by the ``triangular'' vertex-set $\{(i,j) \in [\ell]^2 \setmid i+j \leq \ell+1\}$; see \cref{fig:gridSep}.
    We then attach some simple graphs to $H'_\ell$.
    At each vertex $(i,j) \in V(H'_\ell)$ with $i + j = \ell + 1$ we add $\ell$ pendant edges, making $(i,j)$ the center of a star $K_{1,\ell}$.
    For each vertex $(1,j)$ with $j \in [\ell]$ we add a triangle $C_3$, connecting $(1,j)$ by an edge to one vertex of the triangle.
    And for each vertex $(i,1)$ with $i \in [\ell]$ we add a cycle $C_{2\ell+1}$ of length $2\ell+1$, connecting $(i,1)$ by an edge to one vertex of the cycle.
    See again \cref{fig:gridSep} for an illustration.

    Next, we define the guest graphs.
    For integers $\ell \geq 4$ and $i \in [\ell]$, let $G_i^\ell$ consist of a path $v_0,v_1,\dots,v_{\ell+1}$ on $\ell+2$ vertices with a $C_3$ attached to $v_0$, a $C_{2\ell+1}$ attached to $v_{\ell+1}$, and the center of a star $K_{1,\ell}$ attached to $v_i$; see \cref{fig:gridSep}.
    Then we consider the guest class $\calG = \{G_i^\ell \setmid \ell \in \mathbb{N}, i \in [\ell]\}\cup \{K_2\}$.
    Note that $\calG$ is component-closed, since all graphs in $\calG$ are connected.
    We shall show for every $\ell \geq 1$ that $\cn{l}{\calG}{H_\ell} \leq 2$, while $\cn{u}{\calG}{H_\ell} \geq \ell$, and therefore that $\calH$ is not $(\cn{u}{\calG}{},\cn{l}{\calG}{})$-bounded.
    
    First, consider the injective $\calG$-cover $\varphi \colon G_1^\ell \cupdot \cdots \cupdot G_\ell^\ell \to H_\ell$ of $H_\ell$ illustrated in~\cref{fig:gridSep}.
    In particular, for every $i \in [\ell]$ we have
    \[
        \varphi|_{G_i^\ell}(v_a) = 
            \begin{cases}
                (a,\ell+1-i) & \text{ for } a = 1,\ldots,i\\
                (i,\ell+1-a) & \text{ for } a = i,\ldots,\ell.
            \end{cases}   
    \]
    Observe that $\varphi$ is $2$-local and hence $\cn{l}{\calG}{H_\ell} \leq 2$.
        
    To show that $\cn{u}{\calG}{H_\ell} \geq \ell$, consider any injective $\union{\calG}$-cover $\varphi$ of $H_\ell$, i.e., a cover of $H_\ell$ with vertex-disjoint unions of graphs in $\calG$.
    Our goal is to show that $\varphi$ uses at least $\ell$ graphs from $\union{\calG}$.
    Observe that for every $\ell' \neq \ell$ and every $i' \in [\ell']$ we have $G_{i'}^{\ell'} \not\subseteq H_\ell$, since $C_{2\ell'+1}$ is a subgraph of $G_{i'}^{\ell'}$ but not of $H_\ell$.
    Hence, $\varphi$ uses only copies of $K_2$ and $G_1^\ell,\ldots,G_\ell^\ell$ to cover $H_\ell$.

    Let $v$ be one of $\ell$ the vertices of degree $\ell+2$ in $H_\ell$.
    If the $\ell$ pendant edges at $v$ in $H_\ell$ are not covered by some $G_i^\ell$, then $\varphi$ uses at least $\ell$ different copies of $K_2$, all sharing vertex $v$ and we are done.
    Otherwise, $\varphi$ uses some $G_i^\ell$ to cover the pendant edges at $v$.
    As the $C_{2\ell+1}$ in $G_i^\ell$ is also mapped to $H_\ell$, we necessarily have that $i$ is exactly the distance of $v$ to the nearest copy of $C_3$ in $H_\ell$; i.e., $v$ is the vertex $(i,\ell+1-i) \in V(H_\ell)$.
    As this holds for every vertex of degree~$\ell+2$ in $H_\ell$, we have that $\varphi$ uses each of $G_1^\ell,\ldots,G_\ell^\ell$.
    Crucially, for each $i \in [\ell]$ there is exactly one subgraph of $H_\ell$ that is isomorphic to $G_i^\ell$.
    And for every $j \neq i$, the copy of $G_i^\ell$ in $H_\ell$ shares a vertex with the copy of $G_j^\ell$ in $H_\ell$.
    (It is the vertex $(i,\ell+1-j)$ if $i < j$ and vertex $(j,\ell+1-i)$ if $i > j$.)
    Hence, $\varphi$ uses $\ell$ pairwise intersecting graphs from $\calG$, which implies that $\cn{u}{\calG}{H_\ell} \geq \ell$, as desired.
\end{proof}

\cref{thm:minor_comp_sep,thm:cc_tw_bound,thm:tw_sep} are all concerned with $(\cn{u}{\calG}{},\cn{l}{\calG}{})$-boundedness with respect to non-hereditary guest classes~$\calG$.
To end this section, we consider $(\cn{l}{\calG}{},\cn{f}{\calG}{})$-boundedness, for which the situation turns out quite different.
In fact, we only have a separating example already for the most restricted cases for $\calG$ and $\calH$.

Let us define a \emph{hairy cycle} to be a graph that is obtained from a cycle~$C_n$ by attaching any number of pendant edges to vertices of~$C_n$, see \cref{fig:hairy_cycle} (left) for an example.

\begin{proposition}\label{lem:sep_folded-local_G_cc_H_tw}
    Let~$\calG = \set{G \given \text{$G$ is a hairy cycle}} \cup \{K_2\}$ be the class that consists of~$K_2$ and all hairy cycles and let~$\calS$ be the class of all stars.
    Then we have
    \[
        \cn{l}{\calG}{\calS} = \infty, \qquad \cn{f}{\calG}{\calS} \leq 2.
    \]
    In particular, there exists a component-closed class~$\calG$ and a class~$\calH$ of treewidth~$1$ such that $\calH$ is not $(\cn{l}{\calG}{},\cn{f}{\calG}{})$-bounded.
\end{proposition}
\begin{proof}
    Let~$K_{1,n} \in \calS$ be a star with $n$ leaves and let~$w$ denote the center of~$K_{1,n}$. 
    
    We first argue that~$K_{1,n}$ admits a (non-injective) $2$-local $\calG$-cover, which then certifies $\cn{f}{\calG}{K_{1,n}} \leq 2$.
    If $n=1$, then $K_{1,n}$ consists of a single edge; that is~$K_{1,n} = K_2 \in \calG$. 
    In particular, we have $\cn{f}{\calG}{K_{1,n}} = 1$. 
    If $n \geq 2$, we construct a $2$-local $\calG$-cover of~$K_{1,n}$ as follows.
    Let~$C_4$ be a cycle on four vertices, and $a,b$ be two non-adjacent vertices on $C_4$. 
    Now consider the hairy cycle~$G$ that arises from~$C_4$ by attaching $n-2$~vertices of degree~$1$ to the vertex~$a$.
    Mapping both~$a$ and~$b$ to the center~$w$ of~$K_{1,n}$ and all other vertices to a different leaf of~$K_{1,n}$ yields a non-injective $2$-local $\calG$-cover of~$K_{1,n}$, see \cref{fig:hairy_cycle} for an example.
    
    \begin{figure}[ht]
        \centering
        \includegraphics{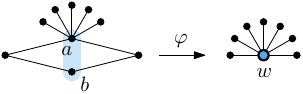}
        \caption{%
            A non-injective $2$-local cover $\varphi$ of the star~$K_{1,7}$ with a hairy cycle.
            Vertices of the hairy cycle that are mapped to the center of the star are colored in blue.
        }
        \label{fig:hairy_cycle}
    \end{figure}

    We now show that $\cn{l}{\calG}{K_{1,n}} \geq n$, which then implies that $\cn{l}{\calG}{\calS} = \infty$.
    Let $\varphi\colon G_1 \cupdot \dots \cupdot G_t \to K_{1,n}$ be any injective $\calG$-cover of~$K_{1,n}$. 
    As $\varphi$ is injective, each guest graph~$G_i$ is a subgraph of the host~$K_{1,n}$. 
    It follows that $G_i \cong K_2$ for all~$i \in [t]$.
    Yet, as each of the $n$ edges of~$K_{1,n}$ is incident to the center~$w$, we have $w \in \varphi(V(G_i))$ for all $i \in [t]$.
    Thus, we have $|\varphi^{-1}(w)| \geq n$ and hence $\cn{l}{\calG}{K_{1,n}} \geq n$.
\end{proof}

\section{Hereditary Guest Classes}
\label{sec:hereditary-guest-classes}

In this section, we consider hereditary (or even monotone) guest classes $\calG$ and in particular investigate when a host class~$\calH$ is $(\cn{u}{\calG}{},\cn{l}{\calG}{})$-bounded, or $(\cn{l}{\calG}{},\cn{f}{\calG}{})$-bounded, or even $(\cn{u}{\calG}{},\cn{f}{\calG}{})$-bounded.
The main advantage of the guest class $\calG$ being monotone (respectively hereditary) is that we can restrict a given $\calG$-cover of a graph~$H \in \calH$ to a $\calG$-cover of any subgraph (respectively weak induced subgraph) of~$H$, as formalized in \cref{lem:restrict-cover-to-subgraph} below.
We call a subgraph $H'$ of~$H$ \emph{weak induced} if each connected component of~$H'$ is an induced subgraph of~$H$.
Note that a weak induced subgraph of $H$ is not necessarily an induced subgraph of $H$.

\begin{lemma}
    \label{lem:restrict-cover-to-subgraph}
    Let $\calG$ be any graph class and $\cn{x}{\calG}{}$ be any $\calG$-covering number.
    Then for every graph $H$ each of the following holds.
    \begin{enumerate}[(i)]
        \item If $\calG$ is monotone, then for every subgraph $H'$ of $H$ we have $\cn{x}{\calG}{H'} \leq \cn{x}{\calG}{H}$. \label{item:cover_subgraph_G_mon}
        \item If $\calG$ is hereditary, then for every induced subgraph $H'$ of $H$ we have $\cn{x}{\calG}{H'} \leq \cn{x}{\calG}{H}$.
        \label{item:cover_ind_subgraph_G_her}
        \item If $\calG$ is hereditary and $\cn{x}{\calG}{} \neq \cn{g}{\calG}{}$, then for every weak induced subgraph $H'$ of $H$ we have $\cn{x}{\calG}{H'} \leq \cn{x}{\calG}{H}$. \label{item:cover_weak_ind_subgraph_G_her}
    \end{enumerate}
\end{lemma}
\begin{proof}
    Let $H_1', \dots, H_k'$ be the components of the subgraph~$H'$ of $H$ and let~$\varphi\colon G_1 \cupdot \dots \cupdot G_t \to H$ be a $\calG$-cover of~$H$. 
    For~$i \in [t]$ and~$j \in [k]$, let~$G_{i,j} = G_{i}|_{H_j'}$ be the subgraph of~$G_i$ when restricting only to those vertices and edges that are mapped to the component~$H_j'$.
    Note that the graphs~$G_{i,\ell}$ and~$G_{i,j}$ are vertex-disjoint for $\ell \neq j$.
    That is, for every~$i \in [t]$, the graph~$G_i' = G_{i,1} \cupdot \dots \cupdot G_{i,k}$ is a subgraph of~$G_i$.
    Restricting~$\varphi$ now yields a cover $\varphi'\colon G_1' \cupdot \dots \cupdot G_t' \to H'$ of~$H'$ that satisfies the following:
    \begin{itemize}
    \item If~$\varphi$ is injective, so is~$\varphi'$.
    \item If~$\varphi$ is $s$-local, so is $\varphi'$.
    \item If~$\varphi$ is $t$-global, so is $\varphi'$.
    \end{itemize}

    If~$\calG$ is monotone, we have~$G_i' \in \calG$ and~$\varphi'$ is in particular a $\calG$-cover of~$H'$.
    Thus, we have $\cn{x}{\calG}{H'} \leq \cn{x}{\calG}{H}$, which concludes the proof of~\eqref{item:cover_subgraph_G_mon}.

    If~$H'$ is an induced subgraph of~$H$, the graphs~$G_i'$ are induced subgraphs of~$G_i$.
    Hence, for hereditary~$\calG$, $\varphi'$ is a $\calG$-cover of~$H'$.
    This shows $\cn{x}{\calG}{H'} \leq \cn{x}{\calG}{H}$ and~\eqref{item:cover_ind_subgraph_G_her} follows.

    If~$H'$ is a weak induced subgraph of~$H$, the graphs~$G_{i,j}$ are induced subgraphs of~$G_i$.
    Hence, if~$\calG$ is hereditary, each~$G_i$ lies in the union-closure~$\overline{\calG}$ of~$\calG$ and~$\varphi'$ is a $\overline{\calG}$-cover of~$H'$, i.e., $\cn{x}{\overline{\calG}}{H'} \leq \cn{x}{\calG}{H}$.
    As $\calG$- and $\overline{\calG}$-covering numbers coincide in the union, local and folded setting (that is $\cn{x}{\overline{\calG}}{H'} = \cn{x}{\calG}{H'}$ for~$\mathrm{x} \in \set{\mathrm{u},\mathrm{l},\mathrm{f}}$), we obtain~\eqref{item:cover_weak_ind_subgraph_G_her}.
\end{proof}

The monotonicity properties in \cref{lem:restrict-cover-to-subgraph} are useful when attempting to bound a more restrictive $\calG$-covering number $\cn{x}{\calG}{}$ in a less restrictive one~$\cn{y}{\calG}{}$.
Suppose we already know another host class~$\calQ$ that is $(\cn{x}{\calG}{},\cn{y}{\calG}{})$-bounded.
Then, given that $\calG$ is hereditary, we can show that also $\calH$ is $(\cn{x}{\calG}{},\cn{y}{\calG}{})$-bounded by decomposing every graph~$H \in \calH$ into few weak induced subgraphs that lie in~$\calQ$.
We use this approach to show two $(\cn{x}{\calG}{},\cn{y}{\calG}{})$-boundedness results:
\begin{itemize}
    \item In \cref{subsec:hereditary-sparse-guests} we use as $\calQ$ the class of all stars and obtain results for hereditary and sparse guest classes; cf.~\cref{thm:union_bounded_by_folded_G_her_mad}.
    \item In \cref{subsec:hereditary-guests-chromatic-hosts} we use as $\calQ$ the class of all bipartite graphs and obtain results for hereditary guest classes combined with host classes of bounded chromatic number; cf.~\cref{thm:union_bounded_by_folded_G_hereditary_H_chi}.
\end{itemize}

\subsection{Hereditary Sparse Guest Classes}
\label{subsec:hereditary-sparse-guests}

Here we consider guest classes $\calG$ that are hereditary (closed under taking induced subgraphs) and \emph{sparse}, that is, have bounded maximum average degree.
Examples of such classes include $k$-planar graphs~\cite{PT97}, $k$-map graphs~\cite{CGP02}, $k$-queue graphs~\cite{HR92}, $k$-stack graphs~\cite{O73,K74,BK79}, genus-$g$ graphs~\cite{GT01}, or any minor-closed class of graphs.

We shall show that with respect to a hereditary and sparse guest class $\calG$, every host class~$\calH$ is $(\cn{u}{\calG}{},\cn{l}{\calG}{})$-bounded and $(\cn{l}{\calG}{},\cn{f}{\calG}{})$-bounded.
As mentioned above, our approach consists of two steps, both centered around the class $\calS = \{K_{1,n} \setmid n \in \N_{\geq 1}\}$ of all stars and the class $\union{\calS}$ of all star forests.
First, we prove that $\calS$ is $(\cn{u}{\calG}{},\cn{l}{\calG}{})$-bounded for every guest class $\calG$, and $\union{\calS}$ is $(\cn{l}{\calG}{},\cn{f}{\calG}{})$-bounded for every hereditary guest class $\calG$.
As every sparse graph is the union of few weak induced star forests \cite{axenovich2019induced}, we then conclude that every host class~$\calH$ is $(\cn{u}{\calG}{},\cn{f}{\calG}{})$-bounded for every hereditary guest class $\calG$.

\begin{lemma}
    \label{lem:folded_equal_union_for_stars_with_G_her}
    Let $\calS = \{K_{1,n} \setmid n \in \N_{\geq 1}\}$ be the class of all stars and $\union{\calS}$ be the class of all star forests.
    Then, for every graph class~$\calG$, each of the following holds.
    \begin{enumerate}[(i)]
        \item For every $F \in \calS$ we have $\cn{l}{\calG}{F} = \cn{u}{\calG}{F}$.\label{item:local-union-star}
        \item If $\calG$ is component-closed, then for every $F \in \union{\calS}$ we have $\cn{l}{\calG}{F} = \cn{u}{\calG}{F}$.\label{item:local-union-star-forest}
        \item If $\calG$ is hereditary, then for every $F \in \union{\calS}$ we have $\cn{f}{\calG}{F} = \cn{l}{\calG}{F}$.\label{item:folded-local-star}
    \end{enumerate}
\end{lemma}
\begin{proof}
    To prove~\eqref{item:local-union-star}, let $F = K_{1,n}$ be a star and let $w$ denote the center (vertex of degree~$n$) of $K_{1,n}$.
    (In case $n=1$, we make an arbitrary but fixed choice for $w$.)
    Consider an $s$-local injective $\calG$-cover $\varphi \colon G_1 \cupdot \cdots \cupdot G_t \to K_{1,n}$ with~$s = \cn{l}{\calG}{K_{1,n}}$.
    It suffices to show that each $G_i$, $i \in [t]$, is incident to the center $w$.
    This is indeed the case as every edge of~$K_{1,n}$ is incident to $w$ and we may assume that each $G_i$ contains at least one edge.
    Thus, no two $G_i$ are vertex-disjoint and hence $\cn{u}{\calG}{K_{1,n}} \leq t = s = \cn{l}{\calG}{K_{1,n}}$.
    
    \smallskip

    To prove~\eqref{item:local-union-star-forest}, let $F \in \union{\calS}$ be a star forest and $\varphi$ be an $s$-local injective $\calG$-cover of $F$ with $s = \cn{l}{\calG}{F}$.
    It is now enough to apply the above reasoning to every component of $F$.
    Indeed, as $\calG$ is component-closed, restricting $\varphi$ to a component $K$ of $F$ gives an $s$-local $\calG$-cover $\varphi_K$ of $K$.
    As argued above, since $K$ is a star, $\varphi_K$ is $s$-global, i.e., we may assume that $K$ is covered in $\varphi_K$ by $s$ graphs $G_{1,K},\ldots,G_{s,K}$ from $\calG$.
    Taking for each $i \in [s]$ the vertex-disjoint union $G_i = \bigcupdot_K G_{i,K}$, which is a graph in $\union{\calG}$, we obtain an $s$-global injective $\union{\calG}$-cover $\varphi' \colon G_{1,K} \cupdot \cdots \cupdot G_{s,K} \to F$ of $F$.
    That is, $\cn{u}{\calG}{F} \leq s = \cn{l}{\calG}{F}$.
    
    \smallskip

    To prove~\eqref{item:folded-local-star}, let $F \in \union{\calS}$ be a star forest with components $F_1,\ldots,F_r \in \calS$.
    As $\calG$ is hereditary (in particular component-closed), we have $\cn{f}{\calG}{F} = \max\{ \cn{f}{\calG}{F_i} \setmid i \in [r]\}$ and $\cn{l}{\calG}{F} = \max\{ \cn{l}{\calG}{F_i} \setmid i \in [r]\}$.
    Hence, it is enough to show that $\cn{f}{\calG}{F_i} = \cn{l}{\calG}{F_i}$ for $i=1,\ldots,r$, i.e., to consider again a single star $K_{1,n}$.
    
    So, consider an $s$-local $\calG$-cover $\varphi \colon G_1 \cupdot \cdots \cupdot G_t \to K_{1,n}$ with~$s = \cn{f}{\calG}{K_{1,n}}$.
    Crucially, note that $\varphi$ need not be injective here.
    As before, let $w$ denote the center of $K_{1,n}$.
    Let $e_1,\ldots,e_n$ denote the edges of $K_{1,n}$ and let $u_i = e_i - w$, $i \in [n]$, be the leaf vertex of $K_{1,n}$ incident to $e_i$.
    For each $i \in [n]$ we (arbitrarily) select one preimage under $\varphi$, i.e., select one edge $e'_i = x_iy_i$ in $G_1 \cupdot \cdots \cupdot G_t$ with $\varphi(x_i) = u_i$ and $\varphi(y_i) = w$.
    As $\varphi^{-1}(u_i)$ is an independent set in $G_1 \cupdot \cdots \cupdot G_t$, the edges $e'_1,\ldots,e'_n$ form a set $\{S_1,\ldots,S_k\}$ of stars in $G_1 \cupdot \cdots \cupdot G_t$.
    Each star $S_j$, $j \in [k]$, is an induced subgraph of some $G_i$, $i \in [t]$.
    Since $\calG$ is hereditary, it follows that $S_j \in \calG$.
    Moreover, each $S_j$ contains a vertex of $\varphi^{-1}(w)$ and hence $k \leq |\varphi^{-1}(w)| \leq s$.

    We now obtain an injective $k$-local $\calG$-cover $\varphi' \colon S_1 \cupdot \cdots \cupdot S_k \to K_{1,n}$ by setting $\varphi'|_{S_j} = \varphi|_{S_j}$ for $j=1,\ldots,k$.
    This shows that $\cn{l}{\calG}{K_{1,n}} \leq k \leq s = \cn{f}{\calG}{K_{1,n}}$, as desired.
\end{proof}

So, \cref{lem:folded_equal_union_for_stars_with_G_her} gives that the class $\union{\calS}$ of all star forests is $(\cn{u}{\calG}{},\cn{f}{\calG}{})$-bounded for every hereditary guest class $\calG$.

Axenovich, Dörr, Rollin and Ueckerdt show in \cite{axenovich2019induced} that every sparse graph is the union of few weak induced star forests.

\begin{lemma}[Axenovich, Dörr, Rollin, Ueckerdt {\cite[Thereom~7]{axenovich2019induced}}]
    \label{lem:mad_bounded_split_into_induced_star_forests}
    If~$H$ is a graph with $\mad(H) \leq d$, then $H$ is the union of at most $2d$~weak induced star forests.
\end{lemma}

In fact, the condition of \cref{lem:mad_bounded_split_into_induced_star_forests} is also met if every induced subgraph of the host graph~$H$ can be covered with few sparse graphs.

\begin{lemma}
\label{lem:mad_G_bounded_bounded_folded_cov_for_induced_subgraphs_then_mad_H_bounded}
    Let~$\calG$ be a graph class with $\mad(\calG) \leq d$ and let~$H$ be a graph.
    If there exists a constant~$s$ such that for every induced subgraph~$H'$ of~$H$ we have $\cn{f}{\calG}{H'} \leq s$, then $\mad(H) \leq sd$.
\end{lemma}
\begin{proof}
    Let $H' \subseteq H$ be any induced subgraph of $H$ with $\abs{V(H')} > 0$ and let~$\varphi' \colon G'_1 \cupdot \cdots \cupdot G'_t \to H'$ be a (not necessarily injective) $s$-local $\calG$-cover of~$H'$.
    As each edge of~$H'$ is covered at least once in $\varphi'$ and each vertex of $H'$ is hit at most $s$~times in $\varphi'$, we obtain
    \[
        2\abs{E(H')} \leq 2\sum_{i=1}^t \abs{E(G'_i)} = \sum_{i=1}^t 2\abs{E(G'_i)} \leq \sum_{i=1}^t \mad(G'_i) \abs{V(G'_i)} \leq \sum_{i=1}^t d \abs{V(G'_i)} \leq s d \cdot \abs{V(H')}.
    \]
    In other words, $\frac{2\abs{E(H')}}{\abs{V(H')}} \leq sd$ for every induced subgraph $H'$ of $H$, and hence $\mad(H) \leq sd$.
\end{proof}

Finally, let us combine \cref{lem:folded_equal_union_for_stars_with_G_her,lem:mad_bounded_split_into_induced_star_forests} to show that for hereditary sparse guest classes $\calG$, every host class~$\calH$ is $(\cn{u}{\calG}{},\cn{f}{\calG}{})$-bounded.
 
\begin{theorem}
    \label{thm:union_bounded_by_folded_G_her_mad}
    If~$\calG$ is a hereditary guest class with $\mad(\calG) \leq d$, then for every graph~$H$ we have
    \[
        \cn{u}{\calG}{H} \leq 2d \cdot (\cn{f}{\calG}{H})^2.
    \]
    In particular, for such guest classes $\calG$, every host class $\calH$ is $(\cn{u}{\calG}{},\cn{f}{\calG}{})$-bounded.
\end{theorem}
\begin{proof}
    Let~$H$ be any host graph and let $s = \cn{f}{\calG}{H}$.
    
    We first bound the mad of~$H$.
    For every induced subgraph~$H'$ of~$H$, we have $\cn{f}{\calG}{H'} \leq s$ by \cref{lem:restrict-cover-to-subgraph}.
    With \cref{lem:mad_G_bounded_bounded_folded_cov_for_induced_subgraphs_then_mad_H_bounded}, $\mad(H) \leq sd$ follows.
        
    By \cref{lem:mad_bounded_split_into_induced_star_forests}, $H$ is the union of at most $k\coloneqq 2(sd)$ weak induced star forests $F_1, \dots, F_k$.
    As each~$F_i$ is a weak induced subgraph of $H$ and $\calG$ is hereditary, \cref{lem:restrict-cover-to-subgraph} gives $\cn{f}{\calG}{F_i} \leq \cn{f}{\calG}{H} = s$ for $i=1,\ldots,k$.
    As~$\calG$ is hereditary, we obtain $\cn{f}{\calG}{F_i} = \cn{u}{\calG}{F_i}$ by \cref{lem:folded_equal_union_for_stars_with_G_her} and 
    \[
        \cn{u}{\calG}{H} \leq \sum_{i=1}^k \cn{u}{\calG}{F_i} =  \sum_{i=1}^k \cn{f}{\calG}{F_i} \leq  2(sd) s = 2d \cdot (\cn{f}{\calG}{H})^2
    \]
    follows.
\end{proof}

Observe that \cref{thm:union_bounded_by_folded_G_her_mad} gives the first claimed $(\cn{u}{\calG}{},\cn{l}{\calG}{})$-boundedness in \cref{thm:union-local-introduction} and the first claimed $(\cn{l}{\calG}{},\cn{f}{\calG}{})$-boundedness in \cref{thm:local-folded-introduction}.

\subsection{Hereditary Guest Classes and Host Classes of Bounded Chromatic Number}
\label{subsec:hereditary-guests-chromatic-hosts}

Here we consider hereditary guest classes $\calG$ combined with host classes $\calH$ of bounded chromatic number.
For example, $\calG$ could be the class of all perfect graphs~\cite{Ber61}, all string graphs~\cite{EET76}, or all $P_4$-free graphs (aka cographs)~\cite{Jun78,Sei74,Sum74}, while we can think of $\calH$ as the class of all $k$-colorable graphs for some fixed $k$.

Our approach is similar to the one in \cref{subsec:hereditary-sparse-guests}, but this time centered around the class $\calB = \{ G \setmid \chi(G) \leq 2\}$ of all bipartite graphs.
First, we prove that $\calB$ is $(\cn{u}{\calG}{},\cn{f}{\calG}{})$-bounded for every hereditary guest class $\calG$.
Second, we show that every $k$-colorable graph is the union of few bipartite graphs.
And finally, we combine both to show that every host class $\calH$ of bounded chromatic number is $(\cn{u}{\calG}{},\cn{f}{\calG}{})$-bounded for every hereditary guest class $\calG$.

\begin{lemma}
\label{lem:union_bounded_by_folded_for_bipartite_host_hereditary_guest}
    For every hereditary guest class~$\calG$ and every bipartite graph~$B$, we have 
    \[
        \cn{u}{\calG}{B} \leq \left(\cn{f}{\calG}{B}\right)^2.
    \]
\end{lemma}
\begin{proof}
    Let $B = (V,E) \in \calB$ be a fixed bipartite graph.
    Consider an $s$-local $\calG$-cover $\varphi \colon G_1 \cupdot \cdots \cupdot G_t \to B$ with $s = \cn{f}{\calG}{B}$.
    Recall that $\varphi$ is not necessarily injective, but each vertex $v \in V$ is hit in total at most~$s$ times by the vertices in $G_1,\ldots,G_t$.
    We denote the preimages of~$v$ in $G \coloneqq G_1 \cupdot \cdots \cupdot G_t$ by $\varphi^{-1}(v) = \{v_1, \ldots, v_r\}$, where $r = \abs{\varphi^{-1}(v)}$.
    Next, we shall construct an injective $s^2$-global $\union{\calG}$-cover of~$B$, which then certifies that indeed $\cn{u}{\calG}{B} \leq s^2$, as desired.
    
    \proofsubparagraph{Construction of an injective $\bm{s^2}$-global $\bm{\union{\calG}}$-cover.} 
    Let~$X$ and $Y$ denote the partition classes of~$B$.
    For each $i \in [s]$, we set $X_i = \{ v_i \in V(G) \setmid v \in X\}$ and $Y_i = \{ v_i \in V(G) \setmid v \in Y\}$.
    Note that $X_1,\ldots,X_s,Y_1,\ldots,Y_s$ form a partition of the vertices in $G = G_1 \cupdot \cdots \cupdot G_t$.
    Clearly, we have~$G \in \union{\calG}$, i.e., the graph $G$ lies in the union-closure of~$\calG$.
    
    For every two integers~$i,j \in [s]$, we define the graph~$G_{i,j}$ as the subgraph of~$G$ induced by the vertex set~$X_i \cup Y_j$.
    As $\calG$ is hereditary, also $G_{i,j}$ is in the union-closure of $\calG$, i.e., $G_{i,j} \in \union{\calG}$.
    We obtain a homomorphism $\varphi' \colon G_{1,1} \cupdot \cdots \cupdot G_{s,s} \to B$ by setting $\varphi'|_{G_{i,j}} = \varphi|_{G_{i,j}}$ for each $i,j \in [s]$.

    \proofsubparagraph{The homomorphism $\bm{\varphi'}$ is an injective $\bm{s^2}$-global $\bm{\union{\calG}}$-cover of $\bm{B}$.}
    We first observe that~$\varphi'$ is a cover of~$B$.
    Consider any edge $uv$ in $B$, say with $u \in X$ and $v \in Y$.
    As $\varphi\colon G \to B$ is a cover of $B$, there exists some edge $u'v'$ in $G$ with $\varphi(u') = u$ and $\varphi(v') = v$.
    Then $u' = u_i$ for some $i \in [s]$ and $v' = v_j$ for some $j \in [s]$.
    It follows that $u' \in X_i$ and $v' \in Y_j$.
    In particular, $u'v'$ is an edge of $G_{i,j}$ and $\varphi'$ maps $u'v'$ onto $uv$.
    
    To see that $\varphi'$ is injective, we observe that each graph~$G_{i,j}$ hits every vertex $v \in V$ at most once. 
    In fact, every vertex~$v \in V$ belongs to only one of~$X,Y$. 
    And as~$X_i$ (and $Y_i$ respectively) contains only one copy of~$v$, at most one vertex in~$G_{i,j}$ is mapped to~$v$.

    Finally, there are only $s^2$ many graphs $G_{i,j}$.
    Hence $\varphi'$ is $s^2$-global.
\end{proof}

So, \cref{lem:union_bounded_by_folded_for_bipartite_host_hereditary_guest} gives that the class $\calB$ of all bipartite graphs is $(\cn{u}{\calG}{},\cn{f}{\calG}{})$-bounded for every hereditary guest class $\calG$.

Now, let us prove that every $k$-chromatic graph is the union of few induced bipartite subgraphs.

\begin{lemma} 
    \label{lem:cover_with_induced_bipartite_graphs}
    If $H$ is a graph with $\chi(H) \leq k$, then $H$ is the union of at most $\binom{k}{2}$ induced bipartite subgraphs.
\end{lemma}
\begin{proof}
    Consider a proper $k$-coloring of~$H$. 
    As every color class is an independent set, any two color classes induce a bipartite subgraph of~$H$.
    These $\binom{k}{2}$ induced bipartite subgraphs clearly partition the edge-set of~$H$.
\end{proof}

Finally, let us combine \cref{lem:union_bounded_by_folded_for_bipartite_host_hereditary_guest,lem:cover_with_induced_bipartite_graphs} to show that for hereditary guest classes $\calG$, every host class~$\calH$ of bounded chromatic number is $(\cn{u}{\calG}{},\cn{f}{\calG}{})$-bounded.

\begin{theorem}
\label{thm:union_bounded_by_folded_G_hereditary_H_chi}
    If $\calG$ is hereditary, then we have for every graph~$H$
    \[
        \cn{u}{\calG}{H}\leq \chi(H)^2 \cdot \cn{f}{\calG}{H}^2.
    \] 
    In particular, for hereditary guest classes $\calG$, every host class $\calH$ of bounded chromatic number is $(\cn{u}{\calG}{},\cn{f}{\calG}{})$-bounded.
\end{theorem}
\begin{proof}
    Let $s \coloneqq \cn{f}{\calG}{H}$ and $k \coloneqq \chi(H)$.
    By \cref{lem:cover_with_induced_bipartite_graphs}, there are $\binom{k}{2} \leq k^2$ bipartite induced subgraphs $B_1,\ldots,B_{k^2}$ of $H$ whose union is $H$.
    As each~$B_i$ is an induced subgraph of~$H$ and~$\calG$ is hereditary, \cref{lem:restrict-cover-to-subgraph} gives $\cn{f}{\calG}{B_i} \leq \cn{f}{\calG}{H} = s$ for $i=1,\ldots,k^2$.
    An application of \cref{lem:union_bounded_by_folded_for_bipartite_host_hereditary_guest} yields $\cn{u}{\calG}{B_i} \leq s^2$ for $i=1,\ldots,k^2$.
    Combining the $s^2$-global $\union{\calG}$-covers of $B_1,\ldots,B_{k^2}$ yields an $(k^2s^2)$-global $\union{\calG}$-cover of~$H$, i.e., $\cn{u}{\calG}{H} \leq k^2s^2$.
\end{proof}

Observe that \cref{thm:union_bounded_by_folded_G_hereditary_H_chi} gives the second claimed $(\cn{u}{\calG}{},\cn{l}{\calG}{})$-boundedness in \cref{thm:union-local-introduction} and the second claimed $(\cn{l}{\calG}{},\cn{f}{\calG}{})$-boundedness in \cref{thm:local-folded-introduction}.

\bigskip

Let us close this section by improving the bound $\cn{u}{\calG}{H} \leq \chi(H)^2 \cn{f}{\calG}{H}^2$ for hereditary guest classes $\calG$ in \cref{thm:union_bounded_by_folded_G_hereditary_H_chi} from quadratic in $\chi(H)$ to logarithmic in $\chi(H)$, provided that $\calG$ is monotone (and not only hereditary).
In fact, recall that by \cref{lem:cover_with_induced_bipartite_graphs} every $k$-chromatic graph~$H$ is the union of at most $\binom{k}{2}$ induced bipartite subgraphs.
Of course, the $\binom{k}{2}$ is best-possible when $H = K_k$.
However, if we do not require the bipartite graphs to be induced subgraphs of $H$, then $\log(k)$ is actually enough (and also best-possible).

\begin{lemma}[{Harary, Hsu, Miller~\cite{Harary1977_biparticity}}] 
    \label{lem:cover_with_log_chi_bipartite_graphs}
    Let~$\calB$ be the class of all bipartite graphs.
    For every graph~$H$, we have
    \[
        \cn{u}{\calB}{H} = \big\lceil\log\chi(H)\big\rceil.
    \]
\end{lemma}

Using \cref{lem:cover_with_log_chi_bipartite_graphs} with monotone guest classes $\calG$, instead of \cref{lem:cover_with_induced_bipartite_graphs} with hereditary guest classes $\calG$, the same argumentation as for \cref{thm:union_bounded_by_folded_G_hereditary_H_chi} gives slightly better bounds.

\begin{proposition}
    \label{thm:union_bounded_by_folded_G_monotone_H_chi}
    If~$\calG$ is monotone, then we have for every graph~$H$
    \[
        \cn{u}{\calG}{H} \leq \big\lceil\log \chi(H)\big\rceil \cdot \cn{f}{\calG}{H}^2.
    \]
\end{proposition}
\begin{proof}
    Let $s \coloneqq \cn{f}{\calG}{H}$ and $k \coloneqq \chi(H)$.
    By \cref{lem:cover_with_log_chi_bipartite_graphs}, there are $x \coloneqq \lceil \log k \rceil$ bipartite subgraphs $B_1,\ldots,B_x$ of $H$ whose union is $H$.
    As each~$B_i$ is a subgraph of~$H$ and~$\calG$ is monotone, \cref{lem:restrict-cover-to-subgraph} gives $\cn{f}{\calG}{B_i} \leq \cn{f}{\calG}{H} = s$ for $i=1,\ldots,x$.
    An application of \cref{lem:union_bounded_by_folded_for_bipartite_host_hereditary_guest} yields $\cn{u}{\calG}{B_i} \leq s^2$ for $i=1,\ldots,x$.
    Combining the $s^2$-global $\union{\calG}$-covers of $B_1,\ldots,B_x$ yields an $(xs^2)$-global $\union{\calG}$-cover of~$H$, i.e., $\cn{u}{\calG}{H} \leq xs^2$.
\end{proof}

\subsection{Hereditary Guest Classes and Host Classes of Unbounded Chromatic Number}
\label{subsec:G-her-H-unbounded-X}

In the previous section, we proved that for every hereditary guest class~$\calG$ any host class~$\calH$ of bounded chromatic number is $(\cn{u}{\calG}{},\cn{f}{\calG}{})$-bounded. 
Here, we shall show that the condition on the chromatic number of the host class~$\calH$ cannot be dropped:
\begin{itemize}
    \item In \cref{subsec:G-her-union-local-separation}, we give an example of a monotone guest class~$\calG$ with bounded chromatic number and a host class~$\calH$ that is not $(\cn{u}{\calG}{},\cn{l}{\calG}{})$-bounded; cf. \cref{lem:sep_local_union_H_mon_G_mon_chi}.
    \item In \cref{subsec:G-her-local-folded-separation}, we give an example of a monotone guest class~$\calG$ with bounded chromatic number and a host class~$\calH$ that is not $(\cn{l}{\calG}{},\cn{f}{\calG}{})$-bounded; cf. \cref{lem:sep_folded-folded_G_mon_chi_H_mon}.
\end{itemize}

\subsubsection{Separating Union and Local Covering Number}\label{subsec:G-her-union-local-separation}

Here we consider the guest class~$\calB$ of all bipartite graphs and the host class~$\calH$ of all shift graphs. 
For a directed graph~$D$, the \emph{shift graph}~$\shift(D)$ is the undirected graph whose vertices correspond to the edges of~$D$ and where two vertices~$uv,xy \in E(D)$ are adjacent if and only if~$v=x$, see \cref{fig:shift_graph_B_v} for an example. 
We shall show that the class of shift graphs~$\calH$ is not $(\cn{u}{\calB}{},\cn{l}{\calB}{})$-bounded.
First, we observe that the local $\calB$-covering number of the class of shift graphs is bounded.

\begin{lemma}
    \label{lem:localShiftCover}
    If~$H$ is a shift graph, then $\cn{l}{\calB}{H} \leq 2$, where $\calB$ denotes the class of all bipartite graphs.
\end{lemma}
\begin{proof}
    Let~$D$ be a directed graph such that $H = \shift(D)$. 
    We construct an injective $2$-local $\calB$-cover of~$H = \shift(D)$, which then certifies $\cn{l}{\calB}{H} \leq 2$.

    For every vertex~$v \in V(D)$, we partition the set of incident edges into the incoming edges~$E_{\mathrm{in}}(v) = \set{uv \in E(D)}$ and outgoing edges~$E_{\mathrm{out}}(v) = \set{vu \in E(D)}$.
    The edges~$E_{\mathrm{in}}(v) \cup E_{\mathrm{out}}(v)$ correspond to vertices in $\shift(D)$ that induce a complete bipartite graph~$B_v$ in~$\shift(D)$, see \cref{fig:shift_graph_B_v} for an example.
    
    \begin{figure}
        \centering
        \includegraphics[page=2]{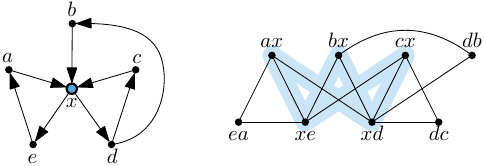}
        \caption{A directed graph~$D$ (left) and the corresponding shift graph~$\shift(D)$ (right). The complete bipartite graph~$B_x \subseteq \shift(D)$ (blue) is induced by the vertices of~$\shift(D)$ corresponding to the incoming and outgoing edges of~$x$.}
        \label{fig:shift_graph_B_v}
    \end{figure}
    
    Now let~$V(D) = \set{1, \dots, n}$.
    The inclusion of the subgraphs~$B_i$ into~$\shift(D)$ yields an injective homomorphism~$\varphi\colon B_1 \cupdot \dots \cupdot B_n \to \shift(D)$.
    
    We first observe that $\varphi$ is a cover of~$\shift(D)$.
    Consider any edge~$e$ in~$\shift(D)$ and let~$uv$ and~$xy$ denote its endpoints.
    Recall that these endpoints correspond to edges in the directed graph~$D$.
    By definition of~$\shift(D)$, we have $v=x$. 
    Thus, the edge~$e$ is covered by the bipartite graph~$B_v$.

    As every vertex~$x$ in~$\shift(D)$ corresponds to an edge~$uv$ in~$D$, the vertex~$x =uv$ is only contained in the cover graphs~$B_u$ and $B_v$.
    Hence, $\varphi$ is $2$-local.
\end{proof}

While the local $\calB$-covering number of shift graphs is small, their chromatic number can be arbitrarily large.

\begin{lemma}[Lov\'asz~{\cite[Problem 9.26 (a)]{lovasz1993combinatorial}}] 
    \label{lem:shiftChrom}
    For every directed graph~$D$ we have $\chi(\shift(D))\geq \log\chi(D)$.
\end{lemma}

As there are shift graphs of arbitrarily large chromatic number, we conclude with \cref{lem:cover_with_log_chi_bipartite_graphs} that the union $\calB$-covering number of the class of shift graphs is unbounded, while the local $\calB$-covering number is bounded (cf. \cref{lem:localShiftCover}).
Hence we have the following.

\begin{proposition}
    \label{lem:sep_local_union_H_mon_G_mon_chi}
    For the class~$\calB$ of all bipartite graphs and the class~$\calH$ of all shift graphs, we have 
    \[
        \cn{u}{\calG}{\calH} = \infty, \qquad \cn{l}{\calG}{\calH} \leq 2.
    \]
    In particular, there is a monotone class~$\calG$ of bounded chromatic number and a host class~$\calH$ that is not ($\cn{u}{\calG}{},\cn{l}{\calG}{})$-bounded.
\end{proposition}

\subsubsection{Separating Local and Folded Covering Number}\label{subsec:G-her-local-folded-separation}
 
Here, we consider the guest class~$\calB$ of all bipartite graphs and the host class~$\calComp$ of all complete graphs.
We shall show that the class~$\calComp$ of all complete graphs is not $(\cn{u}{\calB}{},\cn{f}{\calB}{})$-bounded.
First, we show that the local $\calB$-covering number of complete graphs is unbounded.
Second, we prove that every graph (thus every complete graph) has folded $\calB$-covering number at most $2$. 
It follows that the class~$\calComp$ of all complete graphs is not $(\cn{f}{\calB}{},\cn{l}{\calB}{})$-bounded.

In \cite{FH96}, Fishburn and Hammer study local and global covers with complete bipartite graphs.
In particular, they show that the local covering number of complete graphs with complete bipartite graphs is unbounded.
A similar result for local $\calB$-covering numbers easily follows.

\begin{theorem}[Fishburn-Hammer~{\cite[Theorem~5]{FH96}}]
    \label{thm:FishburnHammer}
    Let~$\calBc = \set{K_{n,m} \given n,m \in \N}$ be the class of all complete bipartite graphs and~$\calB$ the class of all bipartite graphs.
    For the class~$\calComp = \set{K_n \given n \in \N}$ of all complete graphs, we have
    \begin{enumerate}[(i)]
        \item\label{item:local-comp-bip-cover-of-complete} $\cn{l}{\calBc}{\calComp}=\infty$, and
        \item\label{item:local-bip-cover-of-complete} $\cn{l}{\calB}{\calComp}=\infty$.
    \end{enumerate}
\end{theorem}
\begin{proof}
    For a proof of \eqref{item:local-comp-bip-cover-of-complete}, we refer to \cite[Theorem~5]{FH96}.
    
    To prove~\eqref{item:local-bip-cover-of-complete}, we first show that $\cn{l}{\calB}{H} = \cn{l}{\calBc}{H}$ for every complete graph~$H \in \calComp$.
    Consider an injective $k$-local $\calB$-cover~$\varphi\colon B_1 \cupdot \dots \cupdot B_t \to H$ of a complete graph~$H$.
    Adding missing edges, we obtain for each of the graphs~$B_i$, a complete bipartite graph~$B_i'$ on the same vertex set~$V(B_i') = V(B_i)$ that contains~$B_i$.
    This yields an injective $k$-local $\calBc$-cover~$\varphi \colon B_1' \cupdot \dots \cupdot B_t' \to H$ of~$H$.
    In particular, $\cn{l}{\calB}{H} = \cn{f}{\calBc}{H}$ follows.
    With \eqref{item:local-comp-bip-cover-of-complete} we obtain
    $\cn{l}{\calB}{\calComp} = \cn{l}{\calBc}{\calComp} = \infty$ as desired.
\end{proof}

\begin{lemma}
    \label{lem:folded_bipartite_cover}
    Let~$\calB$ be the class of all bipartite graphs.
    For every graph~$H$, we have $\cn{f}{\calB}{H} \leq 2$.
\end{lemma}
\begin{proof}
    For every vertex~$v \in V(H)$, we introduce a \emph{black copy}~$v_{\mathrm{b}}$ and a \emph{white} copy~$v_{\mathrm{w}}$ of~$v$ and call~$v$ the parent of~$v_{\mathrm{b}}$ and~$v_{\mathrm{w}}$.
    Now let~$H'$ be the bipartite graph on the black and white vertices~$V(H') = \set{v_{\mathrm{b}}, v_{\mathrm{w}} \given v \in V(H)}$ where two vertices $u_{\mathrm{b}}$ and~$v_{\mathrm{w}}$ are joined with an edge if $uv\in E(H)$.
    As the graph homomorphism~$\varphi\colon H' \to H$ where every vertex of~$V(H')$ is mapped to its parent in~$V(H)$ is a $2$-local $\calB$-cover, we obtain $\cn{f}{\calB}{H} \leq 2$.
\end{proof}

Combining \cref{thm:FishburnHammer} and \cref{lem:folded_bipartite_cover} now shows that the class~$\calComp$ of all complete graphs is not $(\cn{l}{\calB}{},\cn{f}{\calB}{})$-bounded.

\begin{proposition}\label{lem:sep_folded-folded_G_mon_chi_H_mon}
    For the class~$\calB$ of all bipartite graphs and the class~$\calComp$ of all complete graphs, we have
    \[
        \cn{l}{\calB}{\calComp} = \infty, \qquad \cn{f}{\calB}{\calComp} \leq 2.
    \]
    In particular, there is a monotone guest class~$\calG$ with bounded chromatic number and a host class~$\calH$ that is not $(\cn{l}{\calG}{},\cn{f}{\calG}{})$-bounded.
\end{proposition}

\section{When bounded \texorpdfstring{\bm{$\mathrm{c}_{\mathrm{y}}$}}{cy} implies bounded \texorpdfstring{\bm{$\mathrm{c}_{\mathrm{x}}$}}{cx}}
\label{sec:hereditary_hosts}

In \cref{sec:non-hereditary-guest-classes,sec:hereditary-guest-classes} we investigate whether every host class $\calH$ with certain requirements is $(\cn{x}{\calG}{},\cn{y}{\calG}{})$-bounded for every guest class~$\calG$ with certain requirements.
For any guest class~$\calG$, a host class $\calH$ is \emph{not} $(\cn{x}{\calG}{},\cn{y}{\calG}{})$-bounded if and only if there is a subclass $\calH' \subseteq \calH$ with~$\cn{y}{\calG}{\calH'} < \infty$ but $\cn{x}{\calG}{\calH'} = \infty$.
Let us call $\calH'$ \emph{weakly $(\cn{x}{\calG}{},\cn{y}{\calG}{})$-bounded} if $\cn{y}{\calG}{\calH'} < \infty$ implies that also $\cn{x}{\calG}{\calH'} < \infty$.
In other words, $\calH$ is $(\cn{x}{\calG}{},\cn{y}{\calG}{})$-bounded if and only if every $\calH' \subseteq \calH$ is weakly $(\cn{x}{\calG}{},\cn{y}{\calG}{})$-bounded.

In \cref{sec:non-hereditary-guest-classes,sec:hereditary-guest-classes} we give several examples of $\calG$ and $\calH$ where $\calH$ is not $(\cn{x}{\calG}{},\cn{y}{\calG}{})$-bounded.
The host class $\calH$ is always some ``natural'' graph class (forests, planar graphs, stars, shift graphs, or complete graphs).
However, in two cases (corresponding to \cref{thm:minor_comp_sep,thm:tw_sep}) the subclass $\calH' \subseteq \calH$ that is not weakly $(\cn{x}{\calG}{},\cn{y}{\calG}{})$-bounded consists of very peculiar, ``unnatural'' graphs from $\calH$.
For example, $\calH'$ is not hereditary in these two cases, while the underlying class $\calH$ is hereditary.
In fact, the underlying host classes $\calH$ in \cref{thm:minor_comp_sep,thm:tw_sep} are weakly $(\cn{x}{\calG}{},\cn{y}{\calG}{})$-bounded (for the respective guest class $\calG$), and one might wonder whether there are nicer, namely hereditary, host classes (with the required properties) that are not weakly $(\cn{x}{\calG}{},\cn{y}{\calG}{})$-bounded.

In this section, we consider the question under which restrictions on $\calH$ and $\calG$, we have that~$\calH$ is weakly $(\cn{x}{\calG}{},\cn{y}{\calG}{})$-bounded, i.e., $\cn{y}{\calG}{\calH} < \infty$ implies $\cn{x}{\calG}{\calH} < \infty$.
Every $(\cn{x}{\calG}{},\cn{y}{\calG}{})$-bounded class is also weakly $(\cn{x}{\calG}{},\cn{y}{\calG}{})$-bounded.
So all results in the respective first parts of \cref{thm:local-folded-introduction,thm:union-local-introduction} in particular show weakly $(\cn{x}{\calG}{},\cn{y}{\calG}{})$-boundedness.
But considering the same structural restrictions on $\calG$ and $\calH$ as before, and additionally requiring that~$\calH$ is hereditary, we can prove weakly $(\cn{x}{\calG}{},\cn{y}{\calG}{})$-boundedness in some cases where $(\cn{x}{\calG}{},\cn{y}{\calG}{})$-boundedness does not hold.

\begin{remark*}
    If we do not require $\calH$ to be hereditary, then the situation for weakly $(\cn{x}{\calG}{},\cn{y}{\calG}{})$-boundedness is the same as for $(\cn{x}{\calG}{},\cn{y}{\calG}{})$-boundedness.
    In fact, for every pair $(\calG,\calH)$ with $\calH$ not $(\cn{x}{\calG}{},\cn{y}{\calG}{})$-bounded, there is some $\calH' \subseteq \calH$ such that $\calH'$ is not weakly $(\cn{x}{\calG}{},\cn{y}{\calG}{})$-bounded.
    And all our restrictions on $\calH$ (bounded $\chi$, bounded $\mad$, $M$-minor-free, bounded $\tw$) carry over to $\calH'$.
    Hence, we obtain the pair $(\calG,\calH')$ meeting the same requirements as $(\calG,\calH)$, but with $\calH'$ being not weakly $(\cn{x}{\calG}{},\cn{y}{\calG}{})$-bounded.
\end{remark*}

Let us start with weakly $(\cn{l}{\calG}{},\cn{f}{\calG}{})$-boundedness of hereditary host classes $\calH$.
As before, we ask under which restrictions on $\calG$ and $\calH$, is every hereditary host class $\calH$ weakly $(\cn{l}{\calG}{},\cn{f}{\calG}{})$-bounded with respect to every guest class $\calG$.
In fact, the story here is quickly told, as all results from \cref{sec:hereditary-guest-classes,sec:non-hereditary-guest-classes} simply carry over.
Every $(\cn{l}{\calG}{},\cn{f}{\calG}{})$-bounded $\calH$ is also weakly $(\cn{l}{\calG}{},\cn{f}{\calG}{})$-bounded.
And in all our examples the class $\calH$ that is not $(\cn{l}{\calG}{},\cn{f}{\calG}{})$-bounded is also not weakly $(\cn{l}{\calG}{},\cn{f}{\calG}{})$-bounded.

However, the situation for weakly $(\cn{u}{\calG}{},\cn{l}{\calG}{})$-boundedness of hereditary host classes $\calH$ is different.
Recall that by \cref{thm:minor_comp_sep}, there exists an $M$-minor-free host class $\calH$ and a component-closed guest class $\calG$ such that $\calH$ is not $(\cn{u}{\calG}{},\cn{l}{\calG}{})$-bounded.
While $\calH$ is hereditary (namely the class of all planar graphs), one can show that $\calH$ is weakly $(\cn{u}{\calG}{},\cn{l}{\calG}{})$-bounded.
In fact, we shall show that every hereditary, $M$-minor-free (actually, sparse is enough) host class~$\calH$ is weakly $(\cn{u}{\calG}{},\cn{l}{\calG}{})$-bounded with respect to every component-closed guest class $\calG$.
The argument is similar to the proof of \cref{thm:union_bounded_by_folded_G_her_mad}. 
We decompose each graph in~$\calH$ into a constant number of weak induced star forests and make use of the fact that local and union $\calG$-covering numbers coincide for these graphs (cf. \cref{lem:folded_equal_union_for_stars_with_G_her}).
Recall that a subgraph $H'$ of~$H$ is \emph{weak induced} if each connected component of~$H'$ is an induced subgraph of~$H$.

\begin{proposition}
    \label{lem:union_bounded_by_local_for_H_her_mad}
    Let~$\calH$ be a hereditary host class with $\mad(\calH) \leq d$ and~$\calG$ be any guest class. 
    If there exists a constant~$s$ such that $\cn{l}{\calG}{\calH} \leq s$, then 
    \[
        \cn{u}{\calG}{\calH} \leq 2sd.
    \]
    In particular, such host classes are weakly $(\cn{u}{\calG}{},\cn{l}{\calG}{})$-bounded.
\end{proposition}
\begin{proof}
    Consider any graph~$H \in \calH$. 
    By \cref{lem:mad_bounded_split_into_induced_star_forests}, $H$ is the union of $2d$ weak induced star forests~$F_1, \dots, F_{2d}$.
    For $i \in [2d]$, let $S_{i,1}, \dots, S_{i,k}$ be the components of~$F_i$.
    Each star~$S_{i,j}$ is an induced subgraph of~$H$ and thus lies in~$\calH$, as $\calH$ is hereditary.
    In particular, we have $\cn{l}{\calG}{S_{i,j}} \leq s$ for each star~$S_{i,j}$ and with \cref{lem:folded_equal_union_for_stars_with_G_her}\eqref{item:local-union-star}, we obtain $\cn{u}{\calG}{S_{i,j}} \leq s$.
    It also follows that $\cn{u}{\calG}{F_i} \leq s$.
    That is, for each star forest~$F_i$, there exists an injective $s$-global $\union{\calG}$-cover~$\varphi\colon G_{i,1} \cupdot \dots \cupdot G_{i,s} \to F_i$ of~$F_i$.
    Combining the $2d$ covers of the forests~$F_i$, yields an injective $(s \cdot 2d)$-global $\union{\calG}$-cover~$\varphi\colon G_{1,1} \cupdot \dots \cupdot G_{2d,s} \to H$ which certifies $\cn{u}{\calG}{H} \leq 2sd$.
\end{proof}

In fact, by \cref{lem:mad_G_bounded_bounded_folded_cov_for_induced_subgraphs_then_mad_H_bounded}, the conditions of \cref{lem:union_bounded_by_local_for_H_her_mad} above are also met if we require the guest class to be sparse instead of the host class.

\begin{corollary}
    \label{lem:union_bounded_by_local_for_H_her_G_mad}
    Let~$\calH$ be a hereditary host class and let~$\calG$ be any guest class with~$\mad(\calG) \leq d$. 
    If there exists a constant~$s$ such that $\cn{l}{\calG}{\calH} \leq s$, then $\cn{u}{\calG}{\calH} \leq 2s^2d$.

    In particular, such host classes are weakly $(\cn{u}{\calG}{},\cn{l}{\calG}{})$-bounded.
\end{corollary}

We summarize in \cref{fig:overview_bounded_by_constant} the situation for weakly $(\cn{u}{\calG}{},\cn{l}{\calG}{})$-boundedness of hereditary host classes $\calH$ under our considered structural restrictions on $\calH$ and $\calG$.
Observe that only one separating example remains (namely \cref{lem:sep_local_union_H_mon_G_mon_chi} with $\calH$ being all shift graphs and $\calG$ being all bipartite graphs).
But it remains open whether \cref{lem:union_bounded_by_local_for_H_her_mad} can be generalized from host classes with bounded $\mad$ to host classes with bounded $\chi$.
In fact, it would be enough to resolve the cases of bipartite graphs.

\begin{question}
    Is there a hereditary class $\calH$ of bipartite graphs and a class $\calG$ such that $\calH$ is not weakly $(\cn{u}{\calG}{},\cn{l}{\calG}{})$-bounded, i.e., with $\cn{l}{\calG}{\calH} \leq s$ for a global constant $s$, while $\cn{u}{\calG}{\calH} = \infty$?
\end{question}

\begin{table}[ht]
	\ra{1.3}
	\def\smallDist{3}
	\def\largeDist{8}
	\small
	\centering
	\begin{subtable}[t]{\textwidth}
		\centering
		\begin{tabularx}{0.89\textwidth}{c c c@{\extracolsep{0pt}} @{\extracolsep{\smallDist pt}}c@{\extracolsep{0cm}} @{\extracolsep{\largeDist pt}}c@{\extracolsep{0cm}} @{\extracolsep{\smallDist pt}}c@{\extracolsep{0cm}} @{\extracolsep{\smallDist pt}}c@{\extracolsep{0cm}} @{\extracolsep{\largeDist pt}}c@{\extracolsep{0cm}} @{\extracolsep{\smallDist pt}}c@{\extracolsep{0cm}} @{\extracolsep{\smallDist pt}}c@{\extracolsep{0cm}}}
			\toprule
			\multicolumn{2}{c}{\multirow{2}{*}{\diagbox[width=8em]{\phantom{vu}$\calH$}{$\calG$\phantom{vu}}}} & \multirow{2}{*}{any} & \multirow{2}{*}{cc} & \multicolumn{3}{c}{hereditary} & \multicolumn{3}{c}{monotone} \\
			\arrayrulecolor{black}\cmidrule(r){5-7} \cmidrule(l){8-10}
			& &   &    & \multicolumn{1}{c}{any} & $\chi$ & $\mad$ & \multicolumn{1}{c}{any} & $\chi$ & \multicolumn{1}{c}{$\mad$} \\
			  \cmidrule{3-10}
			\arrayrulecolor{white}
			\multirow{6}{*}{\begin{sideways}hereditary\end{sideways}} & any & \multicolumn{1}{|c|}{\cellcolor{red30}\phantom{\shortref{lem:union_bounded_by_local_for_H_her_mad}}} & \multicolumn{1}{|c|}{\cellcolor{red30}\phantom{\shortref{lem:sep_folded-local_G_cc_H_tw}}} & \multicolumn{1}{|c|}{\cellcolor{red30}} & \multicolumn{1}{|c|}{\cellcolor{red30}\phantom{\shortref{thm:union_bounded_by_folded_G_her_mad}}} & \multicolumn{1}{|c|}{\cellcolor{green30}\shortref{thm:union_bounded_by_folded_G_her_mad}} & \multicolumn{1}{|c|}{\cellcolor{red30}} & \multicolumn{1}{|c|}{\cellcolor{red30}{\shortref{lem:sep_local_union_H_mon_G_mon_chi}}} & \multicolumn{1}{|c|}{\cellcolor{green30}\phantom{\shortref{thm:union_bounded_by_folded_G_her_mad}}} \\
			\cmidrule{3-10}
			& bounded $\chi$ & \multicolumn{1}{|c|}{\cellcolor{openColor}?} & \multicolumn{1}{|c|}{\cellcolor{openColor}?} & \multicolumn{1}{|c|}{\cellcolor{green30}\shortref{thm:union_bounded_by_folded_G_hereditary_H_chi}} & \multicolumn{1}{|c|}{\cellcolor{green30}} & \multicolumn{1}{|c}{\cellcolor{green30}} & \multicolumn{1}{|c|}{\cellcolor{green30}\smallshortref{thm:union_bounded_by_folded_G_monotone_H_chi}} & \multicolumn{1}{|c|}{\cellcolor{green30}} & \multicolumn{1}{|c|}{\cellcolor{green30}}\\
			\cmidrule{3-10}
			& bounded $\mad$ & \multicolumn{1}{|c|}{\cellcolor{green30}\shortref{lem:union_bounded_by_local_for_H_her_mad}} & \multicolumn{1}{|c}{\cellcolor{green30}\phantom{\shortref{lem:union_bounded_by_local_for_H_her_mad}}} & \multicolumn{1}{|c|}{\cellcolor{green30}} & \multicolumn{1}{|c|}{\cellcolor{green30}} & \multicolumn{1}{|c}{\cellcolor{green30}} & \multicolumn{1}{|c|}{\cellcolor{green30}} & \multicolumn{1}{|c|}{\cellcolor{green30}} & \multicolumn{1}{|c|}{\cellcolor{green30}}\\
			\cmidrule{3-10}
			& $M$-minor-free & \multicolumn{1}{|c|}{\cellcolor{green30}} & \multicolumn{1}{|c}{\cellcolor{green30}} & \multicolumn{1}{|c|}{\cellcolor{green30}} & \multicolumn{1}{|c|}{\cellcolor{green30}} & \multicolumn{1}{|c}{\cellcolor{green30}} & \multicolumn{1}{|c|}{\cellcolor{green30}} & \multicolumn{1}{|c|}{\cellcolor{green30}} & \multicolumn{1}{|c|}{\cellcolor{green30}}\\
			\cmidrule{3-10}
			& bounded $\tw$ & \multicolumn{1}{|c|}{\cellcolor{green30}} & \multicolumn{1}{|c|}{\cellcolor{green30}\smallshortref{thm:cc_tw_bound}}& \multicolumn{1}{|c|}{\cellcolor{green30}} & \multicolumn{1}{|c|}{\cellcolor{green30}} & \multicolumn{1}{|c|}{\cellcolor{green30}} & \multicolumn{1}{|c|}{\cellcolor{green30}} & \multicolumn{1}{|c|}{\cellcolor{green30}} & \multicolumn{1}{|c|}{\cellcolor{green30}}\\
			\arrayrulecolor{black}\bottomrule
		\end{tabularx}
        
		\medskip
        
		\caption{%
                Is $\calH$ weakly $(\cn{u}{\calG}{},\cn{l}{\calG}{})$-bounded?%
                \;\; \textcolor{green30}{$\blacksquare$} Always Yes \;\; \textcolor{red30}{$\blacksquare$} Sometimes No \;\; \textcolor{openColor}{$\blacksquare$} Unknown
            }
		\label{fig:union-local_bounded_by_constant}
	\end{subtable}
	\caption{
		Overview of the results in \cref{sec:hereditary_hosts}.
	}
	\label{fig:overview_bounded_by_constant}
\end{table}

\section{Conclusions}
\label{sec:conclusions}

We have considered a specific set of structural properties for a guest class $\calG$ and a host class $\calH$ under which we can conclude that $\calH$ is $(\cn{x}{\calG}{},\cn{y}{\calG}{})$-bounded for a specific pair $\cn{x}{\calG}{}$, $\cn{y}{\calG}{}$ of $\calG$-covering numbers.
Our results are exhaustive; we provide examples of $\calG$ and $\calH$ (having the demanded structural properties) for which $\calH$ is not $(\cn{x}{\calG}{},\cn{y}{\calG}{})$-bounded, whenever such an example exists.
Let us refer again to \cref{fig:overview} for an overview of the results.

Still, some questions remain open.
In the cases that $\calH$ is $(\cn{x}{\calG}{},\cn{y}{\calG}{})$-bounded, one can try to determine the binding function as best as possible.
Our proofs give quadratic binding functions, but it might actually be linear in all cases.
Further, one might consider different requirements for $\calG$ and/or $\calH$, such as being nowhere dense~\cite{nevsetvril_nowhere_dense_2011} or having bounded Weisfeiler-Leman dimension~\cite{WL68}.
Also, some worthwhile graph classes have none or only few of our considered properties.
Consider for example the class $\calG = {\rm Forb}(T) = \{ G \setmid T \not\subseteq G\}$ of all $T$-free graphs for some fixed graph $T$.
This hereditary class is interesting as the global and local ${\rm Forb}(T)$-covering numbers are intimately related to Ramsey numbers and local Ramsey numbers~\cite{GLST87} of $T$, respectively.

As apparent from \cref{fig:overview}, $(\cn{u}{\calG}{},\cn{l}{\calG}{})$-boundedness and $(\cn{l}{\calG}{},\cn{f}{\calG}{})$-boundedness turn out to behave quite similarly in the realm of our investigations.
Is this a coincidence or is there a particular connection to be uncovered?

Finally, whenever a host class $\calH$ is $(\cn{x}{\calG}{},\cn{y}{\calG}{})$-bounded for a guest class $\calG$, one may take advantage of this for algorithmic questions.
For example, for the class $\union{\calS}$ of all star forests, the union $\union{\calS}$-covering number $\cn{u}{\union{\calS}}{H}$ is the star arboricity of $H$ and NP-hard to compute~\cite{HMS96}.
On the other hand, the local counterpart $\cn{l}{\union{\calS}}{H}$ can be computed in polynomial time~\cite{Knauer2016_3w3c1g}.
As $\union{\calS}$ is hereditary (even monotone) and has bounded $\mad$, we can conclude with \cref{thm:union_bounded_by_folded_G_her_mad} that the star arboricity $\cn{u}{\union{\calS}}{}$ can be approximated in polynomial time.
(In fact, for the class $\calF$ of all forests and all graphs $H$ we have $\cn{u}{\union{\calS}}{H} \leq 2\cn{g}{\calF}{H}$~\cite{alon-star-arboricity} and $\cn{g}{\calF}{H} \leq \cn{l}{\union{\calS}}{H}$~\cite{Knauer2016_3w3c1g}, which gives $\cn{u}{\union{\calS}}{H} \leq 2\cn{l}{\union{\calS}}{H}$ and hence a $2$-approximation of $\cn{u}{\union{\calS}}{}$ in polynomial time.)
We leave it to future work to identify further such algorithmic exploits of our results in the present paper.

\bibliographystyle{plainurl}
\bibliography{references}

\begin{thebibliography}{10}

\bibitem{AEH81}
Jin Akiyama, Geoffrey Exoo, and Frank Harary.
\newblock Covering and packing in graphs {IV}: Linear arboricity.
\newblock {\em Networks}, 11(1):69--72, 1981.
\newblock \href {https://doi.org/10.1002/net.3230110108} {\path{doi:10.1002/net.3230110108}}.

\bibitem{AK84}
Jin Akiyama and Mikio Kano.
\newblock Path factors of a graph.
\newblock {\em Graph theory and its Applications}, pages 11--22, 1984.

\bibitem{AA89}
I.~Algor and N.~Alon.
\newblock The star arboricity of graphs.
\newblock {\em Discrete Mathematics}, 75(1):11--22, 1989.
\newblock \href {https://doi.org/10.1016/0012-365X(89)90073-3} {\path{doi:10.1016/0012-365X(89)90073-3}}.

\bibitem{alon-star-arboricity}
Noga Alon, Colin McDiarmid, and Bruce Reed.
\newblock Star arboricity.
\newblock {\em Combinatorica}, 12(4):375--380, 1992.
\newblock \href {https://doi.org/10.1007/BF01305230} {\path{doi:10.1007/BF01305230}}.

\bibitem{axenovich2019induced}
Maria Axenovich, Philip D{\"o}rr, Jonathan Rollin, and Torsten Ueckerdt.
\newblock Induced and weak induced arboricities.
\newblock {\em Discrete Mathematics}, 342(2):511--519, 2019.
\newblock \href {https://doi.org/10.1016/j.disc.2018.10.018} {\path{doi:10.1016/j.disc.2018.10.018}}.

\bibitem{Ber61}
Claude Berge.
\newblock {F\"arbung von Graphen, deren s\"amtliche bzw.\ deren ungerade Kreise starr sind}.
\newblock {\em Wissenschaftliche Zeitschrift}, 1961.

\bibitem{BK79}
Frank Bernhart and Paul~C. Kainen.
\newblock The book thickness of a graph.
\newblock {\em Journal of Combinatorial Theory, Series {B}}, 27(3):320--331, 1979.
\newblock \href {https://doi.org/10.1016/0095-8956(79)90021-2} {\path{doi:10.1016/0095-8956(79)90021-2}}.

\bibitem{blasius2018local_boxicity}
Thomas Bläsius, Peter Stumpf, and Torsten Ueckerdt.
\newblock Local and union boxicity.
\newblock {\em Discrete Mathematics}, 341(5):1307--1315, 2018.
\newblock \href {https://doi.org/10.1016/j.disc.2018.02.003} {\path{doi:10.1016/j.disc.2018.02.003}}.

\bibitem{CGP02}
Zhi-Zhong Chen, Michelangelo Grigni, and Christos~H. Papadimitriou.
\newblock Map graphs.
\newblock {\em Journal of the ACM}, 49(2):127–138, 2002.
\newblock \href {https://doi.org/10.1145/506147.506148} {\path{doi:10.1145/506147.506148}}.

\bibitem{EET76}
G~Ehrlich, S~Even, and R.E Tarjan.
\newblock Intersection graphs of curves in the plane.
\newblock {\em Journal of Combinatorial Theory, Series B}, 21(1):8--20, 1976.
\newblock \href {https://doi.org/10.1016/0095-8956(76)90022-8} {\path{doi:10.1016/0095-8956(76)90022-8}}.

\bibitem{EGP66}
Paul Erd\H{o}s, A.~W. Goodman, and Louis Pósa.
\newblock The representation of a graph by set intersections.
\newblock {\em Canadian Journal of Mathematics}, 18:106–--112, 1966.
\newblock \href {https://doi.org/10.4153/CJM-1966-014-3} {\path{doi:10.4153/CJM-1966-014-3}}.

\bibitem{FH96}
Peter~C. Fishburn and Peter~L. Hammer.
\newblock Bipartite dimensions and bipartite degrees of graphs.
\newblock {\em Discrete Mathematics}, 160(1):127--148, 1996.
\newblock \href {https://doi.org/10.1016/0012-365X(95)00154-O} {\path{doi:10.1016/0012-365X(95)00154-O}}.

\bibitem{gavril1974intersection}
Fǎnicǎ Gavril.
\newblock The intersection graphs of subtrees in trees are exactly the chordal graphs.
\newblock {\em Journal of Combinatorial Theory, Series B}, 16(1):47--56, 1974.
\newblock \href {https://doi.org/10.1016/0095-8956(74)90094-X} {\path{doi:10.1016/0095-8956(74)90094-X}}.

\bibitem{GT01}
Jonathan~L Gross and Thomas~W Tucker.
\newblock {\em Topological graph theory}.
\newblock Courier Corporation, 2001.

\bibitem{FG78}
Andr{\'a}s Gy{\'a}rf{\'a}s and Andr{\'a}s Frank.
\newblock How to orient the edges of a graph?
\newblock {\em Combinatorics}, 18:353--362, 1978.

\bibitem{GLST87}
Andr\'as Gy{\'a}rf{\'a}s, Jen\H{o} Lehel, Richard~H. Schelp, and Zsolt. Tuza.
\newblock Ramsey numbers for local colorings.
\newblock {\em Graphs and Combinatorics}, 3(1):267--277, 1987.
\newblock \href {https://doi.org/10.1007/BF01788549} {\path{doi:10.1007/BF01788549}}.

\bibitem{GW95}
Andr{\'a}s Gy{\'a}rf{\'a}s and Douglas West.
\newblock Multitrack interval graphs.
\newblock {\em Congressum Numeratium}, 139:109--116, 1995.

\bibitem{hajnal1958auflosung}
Andr{\'a}s Hajnal and J{\'a}nos Sur{\'a}nyi.
\newblock {\"U}ber die {A}ufl{\"o}sung von {G}raphen in vollst{\"a}ndige {T}eilgraphen.
\newblock {\em Ann. Univ. Sci. Budapest, E{\"o}tv{\"o}s Sect. Math}, 1:113--121, 1958.

\bibitem{HMS96}
S.L. Hakimi, J.~Mitchem, and E.~Schmeichel.
\newblock Star arboricity of graphs.
\newblock {\em Discrete Mathematics}, 149(1):93--98, 1996.
\newblock \href {https://doi.org/10.1016/0012-365X(94)00313-8} {\path{doi:10.1016/0012-365X(94)00313-8}}.

\bibitem{Harary1977_biparticity}
Frank Harary, Derbiau Hsu, and Zevi Miller.
\newblock The biparticity of a graph.
\newblock {\em Journal of Graph Theory}, 1(2):131--133, 1977.
\newblock \href {https://doi.org/10.1002/jgt.3190010208} {\path{doi:10.1002/jgt.3190010208}}.

\bibitem{HJR85}
Nora Hartsfield, Brad Jackson, and Gerhard Ringel.
\newblock The splitting number of the complete graph.
\newblock {\em Graphs and Combinatorics}, 1(1):311--329, 1985.
\newblock \href {https://doi.org/10.1007/BF02582960} {\path{doi:10.1007/BF02582960}}.

\bibitem{HR92}
Lenwood~S Heath and Arnold~L Rosenberg.
\newblock Laying out graphs using queues.
\newblock {\em SIAM Journal on Computing}, 21(5):927--958, 1992.
\newblock \href {https://doi.org/10.1137/0221055} {\path{doi:10.1137/0221055}}.

\bibitem{Jun78}
H.A Jung.
\newblock On a class of posets and the corresponding comparability graphs.
\newblock {\em Journal of Combinatorial Theory, Series B}, 24(2):125--133, 1978.
\newblock \href {https://doi.org/10.1016/0095-8956(78)90013-8} {\path{doi:10.1016/0095-8956(78)90013-8}}.

\bibitem{K74}
Paul~C. Kainen.
\newblock Some recent results in topological graph theory.
\newblock In {\em Graphs and Combinatorics}, pages 76--108, 1974.
\newblock \href {https://doi.org/10.1007/BFb0066436} {\path{doi:10.1007/BFb0066436}}.

\bibitem{Knauer2016_3w3c1g}
Kolja Knauer and Torsten Ueckerdt.
\newblock Three ways to cover a graph.
\newblock {\em Discrete Mathematics}, 339(2):745--758, 2016.
\newblock \href {https://doi.org/10.1016/j.disc.2015.10.023} {\path{doi:10.1016/j.disc.2015.10.023}}.

\bibitem{Koenig16}
D\'enes K\"onig.
\newblock {\"Uber Graphen und ihre Anwendung auf Determinantentheorie und Mengenlehre}.
\newblock {\em Mathematische Annalen}, 77:453--465, 1916.
\newblock \href {https://doi.org/10.1007/BF01456961} {\path{doi:10.1007/BF01456961}}.

\bibitem{lovasz1993combinatorial}
L{\'a}szl{\'o} Lov{\'a}sz.
\newblock {\em Combinatorial problems and exercises}, volume 361.
\newblock North-Holland, second edition, 1993.

\bibitem{merker2019local_page_numbers}
Laura Merker and Torsten Ueckerdt.
\newblock Local and union page numbers.
\newblock In Daniel Archambault and Csaba~D. T{\'o}th, editors, {\em Graph Drawing and Network Visualization}, pages 447--459, Cham, 2019. Springer International Publishing.
\newblock \href {https://doi.org/10.1007/978-3-030-35802-0_34} {\path{doi:10.1007/978-3-030-35802-0_34}}.

\bibitem{Nash-William1964_decomposition_into_forests}
C.~St.J.~A. Nash-Williams.
\newblock Decomposition of finite graphs into forests.
\newblock {\em Journal of the London Mathematical Society}, s1-39(1):12, 1964.
\newblock \href {https://doi.org/10.1112/jlms/s1-39.1.12} {\path{doi:10.1112/jlms/s1-39.1.12}}.

\bibitem{nevsetvril_nowhere_dense_2011}
Jaroslav Ne{\v{s}}et{\v{r}}il and Patrice~Ossona De~Mendez.
\newblock On nowhere dense graphs.
\newblock {\em European Journal of Combinatorics}, 32(4):600--617, 2011.
\newblock \href {https://doi.org/10.1016/j.ejc.2011.01.006} {\path{doi:10.1016/j.ejc.2011.01.006}}.

\bibitem{O73}
L~Taylor Ollmann.
\newblock On the book thicknesses of various graphs.
\newblock In {\em Proc. 4th Southeastern Conference on Combinatorics, Graph Theory and Computing}, volume~8, page 459. Utilitas Math., 1973.

\bibitem{PT97}
J{\'a}nos Pach and G{\'e}za T{\'o}th.
\newblock Graphs drawn with few crossings per edge.
\newblock {\em Combinatorica}, 17(3):427--439, 1997.
\newblock \href {https://doi.org/10.1007/BF01215922} {\path{doi:10.1007/BF01215922}}.

\bibitem{Rob69}
Fred~S. Roberts.
\newblock On the boxicity and cubicity of a graph.
\newblock In {\em {Recent Progresses in Combinatorics}}, pages 301--310. Academic Press, 1969.

\bibitem{Sei74}
D~Seinsche.
\newblock On a property of the class of n-colorable graphs.
\newblock {\em Journal of Combinatorial Theory, Series B}, 16(2):191--193, 1974.
\newblock \href {https://doi.org/10.1016/0095-8956(74)90063-X} {\path{doi:10.1016/0095-8956(74)90063-X}}.

\bibitem{Sum74}
David~P. Sumner.
\newblock Dacey graphs.
\newblock {\em Journal of the Australian Mathematical Society}, 18(4):492–502, 1974.
\newblock \href {https://doi.org/10.1017/S1446788700029232} {\path{doi:10.1017/S1446788700029232}}.

\bibitem{TH79}
William~T. Trotter~Jr. and Frank Harary.
\newblock On double and multiple interval graphs.
\newblock {\em Journal of Graph Theory}, 3(3):205--211, 1979.
\newblock \href {https://doi.org/10.1002/jgt.3190030302} {\path{doi:10.1002/jgt.3190030302}}.

\bibitem{Tut63}
William~T. Tutte.
\newblock The thickness of a graph.
\newblock In {\em Indagationes Mathematicae (Proceedings)}, volume~66, pages 567--577. Elsevier, 1963.

\bibitem{WL68}
Boris Weisfeiler and Andrei Leman.
\newblock The reduction of a graph to canonical form and the algebra which appears therein.
\newblock {\em Nauchno-Technicheskaya Informatsia}, 2:12--16, 1968.

\end{thebibliography}
\end{document}